\documentclass{article}

\usepackage{PRIMEarxiv}

\usepackage[utf8]{inputenc} % allow utf-8 input
\usepackage[T1]{fontenc}    % use 8-bit T1 fonts
\usepackage{hyperref}       % hyperlinks
\usepackage{url}            % simple URL typesetting
\usepackage{booktabs}       % professional-quality tables
\usepackage{amsfonts}       % blackboard math symbols
\usepackage{nicefrac}       % compact symbols for 1/2, etc.
\usepackage{microtype}      % microtypography
\usepackage{lipsum}
\usepackage{fancyhdr}       % header
\usepackage{graphicx}       % graphics
\graphicspath{{media/}}     % organize your images and other figures under media/ folder

\usepackage{lineno,hyperref}
\usepackage{amsfonts}% per usufruire di mathbb{}%
\usepackage{amsmath}
\usepackage{amsthm}
\usepackage{cases}
\usepackage[shortlabels]{enumitem}

\usepackage{bbm}

\usepackage{tikz,pgf,pgfplots}

\usetikzlibrary{arrows,shapes,snakes,automata,backgrounds,petri,fit,intersections,quotes,positioning}

\usepackage{tikz-3dplot}

\pgfplotsset{compat=1.14}

\usepackage{graphicx}
\usepackage{wasysym}
\usepackage{dsfont} 
\usepackage{mathrsfs}
\usepackage{stmaryrd}
\usepackage{multicol}
\usepackage{cite}
\graphicspath{{graphs/}}

\newcommand\R{\mathbb{R}}
\newcommand\Z{\mathbb{Z}}
\newcommand\N{\mathbb{N}}

\newcommand\Lto{L^2}
\newcommand{\norm}[1]{\left\lVert#1\right\rVert}

\newcommand{\rect}{\textnormal{rect}}
\newcommand{\argmax}{\textnormal{argmax}}

\renewcommand\r{\rangle}
 \renewcommand\l{\langle}
 \newcommand{\dsize}{\displaystyle}

\theoremstyle{plain}
\newtheorem{theorem}{Theorem}[section]

\newtheorem{lemma}[theorem]{Lemma}

\theoremstyle{definition}
\newtheorem{definition}[theorem]{Definition}

\theoremstyle{remark}
\newtheorem{remark}[theorem]{Remark}

\renewcommand\r{\rangle}
 \renewcommand\l{\langle}
\title{A Solution for the Greedy Approximation of a Step Function with a Waveform Dictionary
}

\author{
  Jorge Andres Rivero \\
  Dept. of Economics, University of Washington,\\ 
  Savery Hall, 305 Spokane Ln Box 353330, Seattle, WA, 98105, USA\\
  \texttt{jrivero@uw.edu} \\
   \And
  Pierluigi Vellucci \\
  Dept. of Economics, Roma Tre University, \\
  Via Silvio D'Amico 77, 00145 Rome, Italy. \\
  \texttt{pierluigi.vellucci@uniroma3.it} \\
}

\begin{document}
\maketitle

\begin{abstract}
In this paper we consider a step function characterized by an arbitrary sequence of real-valued scalars and approximate it with a matching pursuit (MP) algorithm. We utilize a waveform dictionary with rectangular window functions as part of this algorithm.  We show that the waveform dictionary is not necessary when all of the scalars are {\color{black} either non positive or non negative} and the parameters of a wavelet dictionary on an integer lattice yields a closed-form solution for the initial optimization problem as part of the MP. Additionally, for any real-valued scalar sequence, we provide a solution {\color{black} with a related wavelet dictionary} at each iteration of the algorithm. {\color{black} This allows for practical calculation of the approximating function, which we use to provide examples on simulated and real univariate time series data that display discontinuities in its underlying structure where the step function can be thought of as a sample from a signal of interest.}
\end{abstract}

% keywords can be removed

\section{Introduction}

In the analysis of signals, important patterns in a process of interest are obscured by noise. In economics, this can be due to liquidity trades, agent errors, unobserved heterogeneity, etc. Managing noise leads to interesting mathematical problems. An example is ``denoising'' a signal that is a step function (piece-wise constant) with many jumps or discontinuities (also known as breaks and singularities) and contaminated by a noisy process. In this context, the commonly held assumption of stationarity is violated, which results in noisy points being falsely characterized as a discontinuity of the signal. This is common in economic and financial time series where stationarity is usually violated in this manner. As an example, \cite{eraker2003} detail the effects of ignoring discontinuities in asset returns, which misleads forecasting and statistical inference on the jump points. 

One approach to the general problem is to consider a sparse representation where fewer parameters reveal pertinent information, perhaps through decomposing signals over elementary waveforms. A popular way to achieve this is by approximating a signal or function of interest by constructing a linear expansion via a collection of elementary functions. As an example, we can consider wavelets which have the property that they can be combined with known portions of an incomplete signal to extract information from the unknown portions. This is possible because the wavelet will correlate with the unknown signal if it contains information of a similar frequency, as demonstrated in a large number of applications: \cite{HAMM20191317,ROY2019162,BRUNI202073,DINC20114602,YAROSHENKO2015265} to name a few. See \cite{STEPHANE20091} for more on sparse representation techniques.

The thrust of this proposal is to provide a flexible approach that allows for simplified calculations while approximating a step function with rectangular functions, inheriting their desirable properties, of which are correctly identifying the discontinuity points. Our approach is through assuming the signal we are interested in is a step function and forming another step function based on a sample generated by this signal. Then we remove noise from the data-derived step function through sub-class of greedy algorithms known as the \emph{Weak-Orthogonal Greedy Algorithm} (WOGA) or \emph{matching pursuit} (MP); see \cite{TEMALYKOV:7} for a mathematical treatment of these algorithms. The algorithm does this by selecting an elementary function from a highly redundant set that is most collinear to the current iterations step function, producing a residual that is carried over to the next step to once again apply the optimization. An execution of the matching pursuit will then produce an approximate orthogonal basis expansion of the signal of interest after a desired number of iterations with these selected elementary functions. The highly redundant set that is used in the MP is referred to as a dictionary and the choice of one is guided by the sought after properties of the signal. We use the waveform dictionary of a rectangular window function, which contains all of the translations, scalings and modulations of the function.

The waveform dictionary contains the wavelet dictionary (system) and the Gabor system providing a flexible collection of functions with desirable properties from both; see \cite{HEIL:15,christensen2003introduction} for their basis and frame counterparts. They have found use in providing stable and reliable approximations based on their elements that may simplify many problems through their redundancy and localization properties; see \cite{DAUBECHIES:2} and results found in \cite{DECARLIchapter,DECARLI2018186} on the Gabor system. In economics, \cite{RAMSEYZHANG:10} applied the waveform dictionary of a Gaussian window function with a matching pursuit algorithm on S\&P 500 index data to support the assertion that the structure of the data is not necessarily that of a random walk. After obtaining a time-frequency representation of the index, they estimate the energy distribution on the time-frequency plane and observe that there are intense clusters in the distribution, which is not consistent with a uniform/random process.  Other applications of waveform dictionaries in economics and finance are discussed in \cite{RAMSEY:11}. 

Our use of the matching pursuit on scalings, translations, and modulations of a rectangular window function provides much flexibility while being simple to interpret. Using time series as an example, the matching pursuit iteratively optimizes over all scalings and translations, which adaptively finds consecutive time periods that are most similar, net the noise. If the true signal is a step function, it will approximate it and also capture its discontinuities or breaks in the representation. However, our framework is not the first to address this issue. Other step function approximation methods are discussed in the survey by  \cite{little2011}, where they notably relate this problem to $k$-means and \cite{prandoni1999} who provide many techniques including local polynomial approaches, which our method can be considered a ``local'' constant approximation. It is important to mention that the iterative nature of our method over a more redundant set than a basis allows more flexibility than these approaches at the cost of complexity.
 
The property of detecting breaks is closely related to the use of wavelets for detecting discontinuities in signals. \cite{grossmann1988} introduces a method to detect discontinuities through considering local derivatives via fine scales of a Hardy wavelet transform. Related to this, \cite{mallat1992} apply the wavelet transform modulus maximum across the domain to detect singularities (discontinuities) and also describe local regularity.  \cite{wang1995} apply a discretized version Daubechies wavelets to detect breaks in stock market returns by considering maximum modulus of expansion coefficients over all shifts with a fixed fine scale. In fact, our approach maximizes the modulus over any scale and shift of rectangular window functions on the raw form of the data, without the added complexity of smoothness. On the note of lack of smoothness, \cite{cattani2004} explores the use of the Haar wavelet basis to detect breaks, while partitioning the data manually. Our approach is a departure from using a basis and user-specified partitioning of any kind. 

Step function signals and wavelet analysis are quite common in economic and financial time series, as well as the problem of detecting breaks. For wavelets, \cite{pal2017time}  investigate the link between Brent crude oil spot price and the food price index; \cite{umar2021impact} analyzed the impact of COVID-19 on the volatility of commodity prices; \cite{bilgili2020estimation} investigated the long and short term impact of biofuel production on the food prices in United States; \cite{mastroeni2022wavelet} introduced a novel measure based on the wavelet transformation and information entropy to investigate oil-food price correlation and its determinants in the domains of time and frequency. For detecting breaks with wavelets, \cite{wong2001} use wavelets to detect breaks in the conditional variance function as well as the signal itself applying it to study nonlinearities between bearish and bullish markets; and \cite{xue2014} develop hypothesis tests of break points using wavelets. A popular method of approximation (estimation) is done through Markov Regime-Switching models where the signal is specified to follow a mixture of Gaussians governed by the break process; see \cite{kim1999} for a comprehensive text. Our approach does not make an assumption about the break process, only that the signal is (approximately) a step function.

We discuss the main results and their implications in Section \ref{mainresult}. In Section 2 we provide the preliminary information: the explanation of the matching pursuit algorithm and lemmas used for our main results. The proof for the main results are found in Section 3. A simple algorithm is discussed in Section 4 and examples with simulated and real data can be found in Section 5.

%%%%%%%%%%%%%%%%%%%%%%%%%%%%%%%%%%%%
%%%%%%%%%%%%%%%%%%%%%%%%%%%%%%%%%%%%
%%%%%%%%%%%%%%%%%%%%%%%%%%%%%%%%%%%%

\subsection{Main Results}\label{mainresult} 

{\color{black} We follow the convention in \cite{RAMSEYZHANG:10} to model the sample of some signal by a sequence of Dirac mass concentrated on points over intervals of integer length.} Let $\{a_j\}_{j=1}^N$ be a {\color{black} real-valued sequence} and let $\rect(x)$ be a constant function equal to 1 with support $[-1/2,1/2]$. Let $\gamma = (t,\xi,u)\in\R_{>0}\times\R^2=\Gamma$, where $\R_{>0}$ is the set of positive real numbers. Define a step function using the sequence as
	\begin{equation}\label{def-f}
		f(x)=\sum_{j=1}^N a_j\rect(x-j), \quad a_j\in \R\, ,
	\end{equation}
and denote the waveform dictionary as
	\begin{equation}\label{eq:wavdict}
		\mathcal{G}(t,\xi,u)=\{G_\gamma\}_{\gamma\in\Gamma}=\bigg\{\frac{1}{\sqrt{t}}\rect\bigg(\frac{x-u}{t}\bigg)e^{2\pi i\xi x}\bigg\}_{(t,\xi,u)\in\Gamma}    \, .
	\end{equation}

{ \color{black}
The waveform dictionary is a large set that contains more information than required. We start by considering the matching pursuit with the set at its largest and then consider a subset. The first iteration of the matching pursuit algorithm is to maximize $\vert\l f, G_\gamma \r\vert$ over $\gamma\in \Gamma$ where $\l\cdot,\cdot\r$ is the usual inner product of functions. The following tells us that whenever the signal is non positive or non negative, then the first iteration of the matching pursuit has a practical closed-form maximum, which is achieved with $\xi = 0$ and integer $t,u$.
}

\begin{theorem}\label{T-positive}
Let $f$ be defined as in (\ref{def-f}), {\color{black} but with a sequence of scalars that are either non positive or non negative}, and $G_{\gamma}$ defined as in (\ref{eq:wavdict}). Then the function $|\l f, G_\gamma\r|$ attains its maximum when $\xi=0$ and $u$ and $t\geq 1$ are integers. Additionally, the maximum value is
	\begin{equation}\label{main-max}
\max_{0\leq k\leq N-1}\ \max_{1\leq n\leq N-k} \max\left\{ \frac{\sum_{j=n}^{n+k}a_{j}}{\sqrt {k+1 }}  , \ \frac{  \sum_{j=n}^{n+k-1}a_j}{\sqrt {k }} , \frac{  \sum_{j=n-1}^{n+k-1}a_j}{\sqrt {k+1}} \right\} 
	\end{equation}
where n and k are integers.
\end{theorem}

Theorem \ref{T-positive} tells us that in the first iteration the waveform dictionary is more redundant than necessary due to $\xi=0$ in the optimum and that a wavelet dictionary is applicable with non negative or {\color{black} non positive} scalars. The matching pursuit is a greedy algorithm that builds a linear expansion of $f$ in terms of a sequence of optimal $G_\gamma$. {\color{black} Therefore, a result like Theorem \ref{T-positive} is appealing since it establishes that finding the maximum is equivalent to finding the subsequence of consecutive sequence elements that have the largest arithmetic average, scaled by the square root of the length of the subsequence.

It is more complicated to solve for $\argmax_\gamma|\langle f, G_{\gamma} \rangle |$ when $f$ is not necessarily a non negative or {\color{black} non positive} function. For one, the modulation term $\xi$ need not be zero and so the approximate linear expansion will be a complex function. This representation in the time-frequency domain approximates and separates frequencies and time effects locally, which enables spectral analysis type applications such as estimating the energy distribution in this domain; see \cite{RAMSEY:11}. However, this additional complexity does not enjoy the simplified form of (\ref{main-max}) and becomes a poor option if one is not interested in both localizations. }{\color{black}In fact, if interest is solely in the approximation of a real-valued signal in the time domain, then complex dictionaries are at best useless since the formulation with a   real-valued dictionary is equivalent provided the dictionary contains conjugate atoms. If the dictionary does not contain these then the resulting approximation is unreliable and can be improved with a real-valued dictionary that provides a good approximation in the least squares sense.}

{\color{black} The second issue, as will be seen in Section 4, is in calculating the expansion. This requires iterative applications of this maximization problem since the first iteration will necessarily create a residual sequence for the next iteration that includes negative and positive terms, {\color{black} regardless if the original sequence was of one sign. Therefore Theorem \ref{T-positive} can't be used to build an algorithm that simplifies this version of matching pursuit, but it does improve the first step if the condition is met.}

We now consider a wavelet dictionary $\mathcal{G}(t,0,u) = \{G_{\widetilde{\gamma}}\}$ with rectangular window and show that the maximum value in this constraint set is similar to the form of (\ref{main-max}) with any arbitrary real-value sequence.

\begin{theorem}\label{T-gen-wavelet}
Let $f$ be defined as in (\ref{def-f}) and $G_{\widetilde{\gamma}} = G_{(t,0,u)}$ defined as in (\ref{eq:wavdict}). Then the function $|\l f, G_{\widetilde{\gamma}}\r|$ attains its maximum when $u$ and $t\geq 1$ are integers and its maximum value is given by
	\begin{equation}\label{main-max-2}
\max_{0\leq k\leq N-1}\ \max_{1\leq n\leq N-k} \max\left\{ \frac{\vert \sum_{j=n}^{n+k}a_{j} \vert}{\sqrt {k+1 }}  , \ \frac{ \vert \sum_{j=n}^{n+k-1}a_j \vert}{\sqrt {k }} , \frac{\vert \sum_{j=n-1}^{n+k-1}a_j \vert}{\sqrt {k+1}} \right\} 
	\end{equation}
where n and k are integers.
\end{theorem}

Therefore selecting from the wavelet dictionary leads to the optimal value at each iteration of the matching pursuit, simplifying calculations greatly while incurring a cost excluding local frequency representation. The representation from using the wavelet dictionary is useful to approximate and separate time effects of the signal. This dictionary using the rectangular window will well represent functions that are (approximately) piecewise constant since it matches the functions property of time localization. Even with the simple rectangular window you can achieve good approximations since it can be interpreted as placing individual Dirac masses over the Dirac mass derived from the sample.

 In Section 4 we discuss the computation and the limitations of the step-function approximation using the waveform dictionary as well as introduce a simple algorithm leveraging (\ref{main-max-2}) that calculates the representation in polynomial time. In section 5 we provide some examples of the wavelet version of this matching pursuit with simulated and real data. In particular, we highlight how the discontinuities of the wavelet dictionary with rectangular window functions is well-adapted to time series with several breaks and relate this approach to $k$-means clustering.}

%%%%%%%%%%%%%%%%%%%%%%%%%%%%%%%%%%%%
%%%%%%%%%%%%%%%%%%%%%%%%%%%%%%%%%%%%
%%%%%%%%%%%%%%%%%%%%%%%%%%%%%%%%%%%%

\section{Preliminaries}\label{prelims}
\label{sec:prel}
We denote the set of nonnegative integers as $\N$. We denote the Hilbert space of real-valued functions that are square integrable in the Lebesgue sense by 
	\begin{equation*}
    		L^2(\mathbb R)=\left\{f: \int_\R |f(t)|^2 dt<+\infty\right\}
	\end{equation*}
with inner product and norm that, on $\mathbb R$, are
	\begin{equation*}
   		 \langle f, g \rangle=\int_\R f(t)g(t) dt,\qquad \norm{f}_2=\sqrt{\langle f, f \rangle}\, .
	\end{equation*}
All of the norms will be the $L^2$-norm on $\mathbb R$ hereafter. In the following subsections we formally define the waveform dictionary, present the matching pursuit algorithm, and prove the lemmas that we use for our main results.

%%%%%%%%%%%%%%%%%%%%%%%%%%%%%%%%%%%%
%%%%%%%%%%%%%%%%%%%%%%%%%%%%%%%%%%%%
%%%%%%%%%%%%%%%%%%%%%%%%%%%%%%%%%%%%

\subsection{The Waveform Dictionary and the Matching Pursuit}\label{matchpursuit}

The matching pursuit can be used with any dictionary, however the choice is important. The dictionary determines the functional form that will be approximated, which emphasizes features of the data, and the computational cost of the algorithm. For example, it is known that the computational complexity of the matching pursuit is to the order of an exponential when utilizing the waveform dictionary with a Gaussian window; see~\cite{DMA:9}. We use a waveform dictionary in our matching pursuit.
\begin{definition}
\label{def:waveform_dictionary}
Let $\gamma=(t,\xi,u)\in\Gamma=\R_{>0} \times\R^2$. A \textit{waveform dictionary} $\mathcal{G}$ is a collection of $\Lto(\R)$ functions of the form
\begin{equation*}
G_\gamma(x)=\frac{1}{\sqrt{t}}g\bigg(\frac{x-u}{t}\bigg)e^{2\pi i \xi x}\, ,
\end{equation*}
where $g\in\Lto(\R)$ is called the \emph{window function} that satisfies $\norm{g}_{\Lto(\R)}=1$, $g(0)\neq0$, and the integral of $g$ is nonzero. The function $G_\gamma$ is known as a \emph{time-frequency atom}.
\end{definition}
This dictionary is a redundant set of the modulations, translations, and dilations of a window function $g$. {\color{black} In the subset of $\Gamma$ where $\xi= 0$, we say $\mathcal{G}$ is a \emph{wavelet dictionary}} that comprises translations and dilations of $g$. As previously stated, the choice of $g$ is discretionary and based on what properties of the data set will be reflected in the representation; see \cite{RAMSEY:11} for an example in financial time series where $g$ is chosen to be a Gaussian window function. An advantage of using this dictionary is that it is flexible with functions that are well localized in time and frequency through the mixture of dilation and modulation terms. 

Consider an $N$-term approximation of a function $f\in\Lto(\R)$ using waveform dictionary elements 
	\begin{equation}
		\widehat{f}= \sum_{j=1}^{N} b_j G_{\gamma_j}\, ,    
	\end{equation}
where $b_j\in\R$ and the $b_j$ and $\gamma_j$ are chosen such that $||f-\widehat{f}||_2$ is minimized. We use the matching pursuit to successively apply orthogonal projections of $f$ onto the dictionary elements $G_{\gamma_j}$ to find such $b_j$ and $\gamma_j$. 

Choosing such a $G_{\gamma_0}\in\mathcal{G}$ in our initial iteration, we may decompose $f$ as the so-called \emph{1st order greedy expansion}
	\begin{equation}
		f=\langle f, G_{\gamma_0}\rangle G_{\gamma_0} + Rf\, ,
	\end{equation}
where $Rf$ is the residual after approximating $f$ in the direction of $G_{\gamma_0}$. We have that $G_{\gamma_0}$ is orthogonal to $Rf$ so that
	\begin{equation}
		\norm{f}_2^2= |\langle f, G_{\gamma_0}\rangle|^2 + \norm{Rf}_2^2\, .
	\end{equation}
Observe that
	\begin{equation}
   		 \norm{Rf}_2^2= \norm{f}_2^2 -|\l f, G_{\gamma_0}\r |^2
	\end{equation}
and so $||Rf||_2$ is minimum if and only if $|\l f, G_{\gamma_0}\r |^2$ is maximum. Therefore, we must choose $\gamma_0$ that maximizes $|\langle f, G_{\gamma_0}\rangle|$ or at best the $\gamma_0$ must satisfy
	\begin{equation}\label{conv}
		|\langle f, G_{\gamma_0}\rangle|\geq \alpha_0 \sup_{\gamma\in\Gamma}|\langle f, G_{\gamma}\rangle|\, ,
	\end{equation}
where $0\leq\alpha_0\leq 1$. The last case  applies when the maximum is difficult and unnecessary to calculate, which depends on application of the matching pursuit. This concludes the first iteration and in the following iteration we repeat the procedure with $Rf$ in place of $f$.

We explain inductively the mechanism of the matching pursuit beyond the initial iteration. Suppose that the $n$th order residual $R^nf$ has been found. We choose $\gamma_n\in\Gamma$ in order to maximize $|\langle R^n f, G_{\gamma_n}\rangle|$ or satisfy
	\begin{equation}\label{convmain}
		|\langle R^n f, G_{\gamma_n}\rangle|\geq \alpha_n \sup_{\gamma\in\Gamma}|\langle R^n f, G_{\gamma}\rangle|\, .
	\end{equation}

The residual $R^n f$ is then decomposed as 
	\begin{equation}\label{eq:2.5}
		R^n f = \langle R^n f, G_{\gamma_n}\rangle G_{\gamma_n} + R^{n+1}f  
	\end{equation}
which gives us a relationship between the $n$th order residual and the $(n+1)$-th order residual. Since $R^{n+1}f$ is orthogonal to $G_{\gamma_n}$,
	\begin{equation}
		\norm{R^nf}_2^2=\vert\langle R^n f, G_{\gamma_n}\rangle\vert^2 + \norm{R^{n+1}f}_2^2\, .
	\end{equation}
We will continue this decomposition until order $m\geq n$. Then $f$ can be represented as the concatenated sum
	\begin{equation}\label{eq:2.7}
		f= \sum_{k=0}^{m-1} (R^k f - R^{k+1} f ) +R^m f\, 
	\end{equation}
and combining (\ref{eq:2.5}), (\ref{eq:2.7}), we get the $(m-1)$-th order greedy expansion of $f$:
\begin{equation}
f=\sum_{k=0}^{m-1}\langle R^k f, G_{\gamma_k}\rangle G_{\gamma_k} + R^m f\, .
\end{equation}

As we continue this procedure we obtain a sequence $\{\alpha_0,\dots,\alpha_n,\dots\}$ of constants from (\ref{convmain}); of functions of the dictionary $\{G_{\gamma_0}, G_{\gamma_1},\dots,G_{\gamma_n},\dots\}$; and a sequence of functions $\{f_1,\dots,f_n,\dots\}$ defined by
\begin{equation}
    f_m= f-\sum_{j=0}^{m-1} \l f_j, G_{\gamma_j}\r G_{\gamma_j}
\end{equation} 
for which $ ||f_m||_2$ is as small as possible. In order for  $ \norm{f_m}_2$ to converge to zero, the sequence $\{\alpha_m\}$ must satisfy the conditions as part of the following theorem.
\begin{theorem}\label{convtheorem}
In the class of monotone sequences $\tau=\{t_k\}_{k=1}^\infty$ such that $0\leq t_k\leq 1$, for every $k=1,2,\dots$, the condition 
	\begin{equation} \label{def-converge}
		\sum_{k=1}^\infty \frac{t_k}{k}=\infty
	\end{equation}
is necessary and sufficient for convergence of the Weak Greedy Algorithm (matching pursuit) for each f in any Hilbert space and for any dictionary with elements $\varphi_j$, and
	\begin{equation}
		f=\lim_{n\to\infty}\sum_{j=1}^n\l f_{j-1}^\tau,\varphi_j^\tau\r\varphi_j^\tau \, .
	\end{equation}
\end{theorem}

The proofs of Theorem \ref{convtheorem} are found in \cite{TEMLYAKOV:16} for necessity and \cite{LIVSHITZTEM:18} for sufficiency. The theorem says that if we must settle for small $\alpha$ ($\tau$ in the theorem) in condition (\ref{convmain}), the algorithm may incur a loss of convergence. 

{\color{black}
For the approximation using a waveform dictionary, Theorem \ref{T-positive} does not provide the maximized value for any step beyond the initialization because those residual terms $R^{n}$ for $n>0$ necessarily will contain sequence elements of mixed sign. Therefore, the condition (\ref{convmain}) must be relied upon to deliver convergence guarantees using the potentially suboptimal form (\ref{main-max}), which we do not attempt. However, with a wavelet dictionary, Theorem \ref{T-gen-wavelet} provides the maximum value of each iteration so there is equality in (\ref{convmain}). Then, the sequence of $\{\alpha_m\}$ can be chosen for which (\ref{def-converge}) is the harmonic series giving a convergence guarantee of the algorithm with a wavelet dictionary. 
}

{\color{black}
\begin{remark}
When the scalars $a_j$ are not necessarily  non positive or non negative, an optimal $\xi$ for $\max_{\gamma\in\Gamma}|\langle f,\, G_{\gamma}\rangle|$ can be non zero. Therefore this will result in an approximate complex representation when the data is real-valued. In Definition \ref{def:waveform_dictionary}, the basis function $G_\gamma(\cdot)$ is centered at $u$ and its energy, that is proportional to $t$, is concentrated in a neighborhood of $u$. On the other hand, the Fourier transform of $G_\gamma(\cdot)$ is centered at $\xi$ with energy in a neighborhood of $\xi$ that is proportional to $1/t$. In other words, the choice of a complex-valued function like the $G_\gamma(\cdot)$ allows us to select the appropriate trade-off between resolution in the time and frequency domains, where using a specialized basis may not apply throughout the domain. 

The complex-valued representation via waveforms also takes the information in the signal to the time-frequency domain allowing for estimation of the energy distribution for time-frequency pairs. This allows to analyze patterns of the signal in both these dimensions, see \cite{RAMSEYZHANG:10} for an example. In general, these sorts of transforms are useful for spectral analysis to extract the frequency information from the signal, {\color{black} but not necessarily useful for a representation in the time domain where one is better off using a real-valued dictionary.} As an example, the Morlet wavelet as a window function finds use in applications when it is important to understand if, given two signals, one of them drives the other. Since the Morlet wavelet is a complex function, it is possible to write the wavelet cross spectrum in terms of its phase and modulus and, according to the value of the wavelet phase difference, it is also possible to establish if the series co-move and which series is leading the other one \cite{umar2021impact,mastroeni2022wavelet,pal2017time,bilgili2020estimation}.
\end{remark}
}

%%%%%%%%%%%%%%%%%%%%%%%%%%%%%%%%%%%
%%%%%%%%%%%%%%%%%%%%%%%%%%%%%%%%%%%
%%%%%%%%%%%%%%%%%%%%%%%%%%%%%%%%%%%
\subsection{Useful Lemmas}

The following lemma serves as a model case for Theorem \ref{T-positive} and \ref{T-gen-wavelet}. 

\begin{lemma}\label{L-model-case}
Let $0<t\leq1$ and let $n_0$ be such that $|a_{n_0}|=\max_{0\leq j\leq N} \{|a_j|\}$. Let $f$ be defined as in (\ref{def-f}) and $G_{\gamma}$ as in (\ref{eq:wavdict}). Then
	\begin{equation}
		\max_{\gamma=(t,\xi,u)\in \Gamma} |\langle f,\, G_{\gamma}\rangle| = |a_{n_0}|
	\end{equation}
which is attained when $(t,\xi,u) = (1,0,n_0)$.
\end{lemma}

\begin{proof}
Fix $u\in\R$. Let  $n \in\Z$ be  such that
\begin{equation}\label{e-u} 
n-\frac 12\leq  u-\frac t2 \leq n+\frac 12 . 
\end{equation} 
Let ${\displaystyle s= n -u +\frac{t}{2} - \frac{1}{2}}$. Note that $0\leq s\leq t$; see Figure \ref{fig1}.
 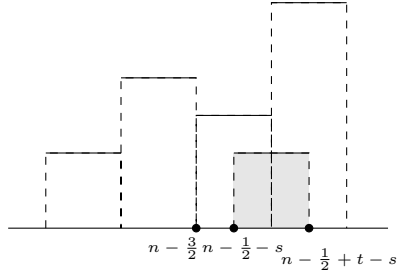
\begin{figure}[hbt!]
\centering
\begin{tikzpicture}
 \fill[gray!20] (1.5,0)--  (1.5, 1)-- (2.5,1)--(2.5, 0);
\draw[black, dashed] (1.5,0)--  (1.5, 1)-- (2.5,1)--(2.5, 0);
 \draw[black] (-1.5,0)--(3.55,0);
  \draw [black, dashed ] (-1,0)--  (-1, 1)--(0,1)--(0, 0)--(0, 2)--(1, 2)--(1,0)--(1, 1.5)--(2,1.5)--(2,0)--(2, 3)--(3,3)--(3,0) ;
   \draw[black] (1.5, 1)-- (2.5,1);
   \draw[black] (-1, 1)-- (0,1);
   \draw[black](0, 2)--(1, 2);
   \draw[black] (1, 1.5)--(2,1.5);
   \draw[black] (2, 3)--(3,3);
 \draw [black,fill]  (1.5,0) circle [radius=0.05];
 \draw (1.62,0) node [black,    below]   {{\tiny $n-\frac 12\!-s $}};
 \draw [black,fill]  (2.5,0) circle [radius=0.05];
 \draw (2.9, -.15) node [black,    below] {{\tiny $ n-\frac 12+t-s $}};
 \draw [black,fill]  (1,0) circle [radius=0.05];
  \draw (.7,0) node [black,    below] {{\tiny $ n-\frac 32 $}};
  \end{tikzpicture}
    \caption{$0<t\leq 1$, the shaded region is the area under the window $G_\gamma$. The steps are of $f$.}
    \label{fig1}
 \end{figure}

We begin by calculating $|\l f, G_\gamma \r |$:
\begin{align}
 \vert\langle f,\, G_{\gamma}\rangle\vert &= \bigg\vert\int_\R \bigg[\sum_{j=1}^Na_j\rect(x-j)\bigg]\bigg[\frac{1}{\sqrt{t}}e^{2\pi i\xi x}\rect\bigg(\frac{x-u}{t}\bigg)\bigg]\, dx\bigg\vert \\ \label{2}
&= \frac{1}{\sqrt{t}}\bigg\vert\sum_{j=1}^Na_j\int_\R \rect(x-j)\rect\bigg(\frac{x-u}{t}\bigg)e^{2\pi i\xi x}\, dx\bigg\vert\\ \label{form}
&= \frac{1}{\sqrt{t}}\bigg\vert a_{n-1}\!\int_{n-\frac 12-s}^{n - \frac{1}{2}}e^{2\pi i \xi x}\, dx + a_{n} \int_{n - \frac{1}{2}}^{n-\frac 12+t-s}e^{2\pi i \xi x} \, dx\bigg\vert\\
   &= \frac{1}{\sqrt{t}}\left|\frac{e^{2\pi i \xi(n-\frac 12)}}{2\pi i \xi }\left( a_{n-1}(1-e^{-2\pi i \xi s})  + a_{n }( e^{ 2\pi i \xi (t-s)}-1)\right)\right|
\\ 
&=\frac{1}{ \pi  |\xi|\sqrt{t}}\left| a_{n-1} e^{ -\pi i s\xi}\sin(\pi \xi s) + a_{n }e^{ \pi i (t-s)\xi}\sin(\pi \xi (t-s))\right|\\ 
&= \frac{1}{ \sqrt{t}}\left| a_{n-1}  \frac{\sin(\pi \xi s)}{\pi  \xi} + a_{n }e^{ \pi it\xi}\frac{\sin(\pi \xi (t-s))}{\pi  \xi}\right|.
\end{align}
where the second to last equality is given by the identity ${\displaystyle \sin(x) = \frac{e^{ix} - e^{-ix}}{2i}}$.

Then we define the function
\begin{equation}\label{def-h2}   h(\xi,\,s,\,t) =\begin{cases} \frac{1}{ \sqrt{t}}\left| a_{n-1}  \frac{\sin(\pi \xi s)}{\pi  \xi} + a_{n }e^{ \pi it\xi}\frac{\sin(\pi \xi (t-s))}{\pi  \xi}\right| &\mbox{ if  $\xi \ne 0$} \cr\cr
  \dsize \frac{  |a_{n-1}s+a_{n } (t-s )|  }{  \sqrt t} & \mbox{ if  $\xi = 0$} \cr\end{cases}
\end{equation}
and show that
 \begin{equation}\label{e-max-2}  
 \max_{{\xi\in\R}\atop{0\leq  s\leq t\leq 1} }  h(\xi, s,t) =  \max\{|a_n|, \ |a_{n-1} |\} .  
 \end{equation}
Firstly we can see that
\begin{equation}\label{part1} 
\max_{0\leq s\leq t\leq 1} h(0,s,t)= \max_{0\leq s\leq t\leq 1}\frac{  |a_{n-1}s+a_{n } (t-s )|  }{  \sqrt t}= \max\{|a_n|, \ |a_{n-1} |\}
\end{equation}  
by noticing that the function  $s\mapsto |a_{n-1}s+ a_{n } (t-s)|$ is  piecewise linear in $[0, t]$ and so it attains its maximum either when $s=0$ or $s=t$. Consequently, we will have that 
	\begin{equation}
		 \max_{\xi\in\R,0\leq  s\leq t\leq 1 }  h(\xi, s,t)\ge \max\{|a_n|, \ |a_{n-1} |\}.
	\end{equation}

To conclude the proof for (\ref{e-max-2}), consider the following inequality where $\xi\neq0$, 
 \begin{align}
 h(\xi, s, t) &=  \frac 1{\sqrt t} \left|a_{n-1} \,  \frac{\sin(\pi \xi s)}{ \pi\xi }+ a_{n }  \frac{\sin(\pi \xi (t-s))}  { \pi\xi   }  \right| \\
 &\leq \frac 1{\sqrt t} (|a_{n-1}|s+ |a_{n }| (t-s  )). 
 \end{align}
Using the same argument as before, we see at once that
 	\begin{equation}
   		\max_{0\leq s\leq t\leq 1} \frac{  |a_{n-1}|s+ |a_{n }| (t-s )   }{  \sqrt t} = \max\{  |a_{n}|,\   |a_{n-1}|\}.
   	\end{equation}
and (\ref{e-max-2}) is proved.

We have shown that 
\begin{equation}
\max\{|a_n|, \ |a_{n-1} |\}\leq\dsize \max_{\xi\in\R,0\leq  s\leq t\leq 1 }  h(\xi, s,t)\leq \max\{|a_n|, \ |a_{n-1} |\}
\end{equation}
and by the definition of $h$,
\[
 	\max_\gamma |\l f,\  G_\gamma\r| \leq  \max_{u}|\l f,\  G_{(1, 0, u)}\r| = \max\{  |a_{n}|,\  | a_{n-1}|\}  
\]
where $n \in\Z$ is such that $n \leq  u  \leq n+1$. 
 If we choose $u=n_0$  as in the statement of the lemma, we have that 
 $\max_{\gamma=(t,\xi,u)\in (0,1]\times\R^2}  |\l f,\  G_\gamma\r| \leq  |\l f,\  G_{(1, 0, n_0)}\r| =   |a_{n_0}|$ as required. 
\end{proof}

{\color{black}

Let  $a_j \in \R$ for all $j= 1, \dots, N$, fix $n = 1,\dots, N$ and let
	\begin{equation}\label{e-defpsi1}
		\psi_k(s,t)=\frac{1}{\sqrt{t}}\left\vert a_{n-1}s + (a_{n }+... + a_{n+k-1}) + a_{n+k }  (t-s-k )\right\vert
	\end{equation}
where $s$ is defined as before.

The following lemma shows the maximum of the function $\psi_k(s,t)$ is in the triangle $T_k$ formed by the lines $s=0$, $s=t-k $ and $t=k+1$, with integer $k\geq 0$.
\begin{lemma}\label{Lk} 
Let $\psi_k(s,t) $ as in (\ref{e-defpsi1}). Then
\begin{equation}\label{e-mk}
 \max_{(s,t) \in T_k} \psi_k(s,t) = \max\{\psi_k(0,\ k),\ \psi_k(0,\ k+1),\ \psi_k(1,\ k+1)\} \, .
\end{equation}
\end{lemma}

\noindent Note that this is true for any sequence including those applicable to Theorem \ref{T-positive}. 

\begin{proof}

First we emphasize that $a = a_{n-1}$, $b = a_{n }+... + a_{n+k-1}$, and $c = a_{n-k}$ are scalars. If $a = c = 0$ then (\ref{e-mk}) holds. Suppose at least one of $a \neq 0$ or $c \neq 0$ holds. Then note that the function $s \mapsto \psi_k(s,t)$ is a convex transformation of an affine function, hence it is convex. Since it is also continuous and piece-wise linear it is maximized strictly on the boundary of $0\leq s \leq t - k$ and
\[
	\psi_k(s,t) \leq \max\{\psi_k(0,t), \psi_k(t-k,t)\}.
\]
Then, the functions $t \mapsto \psi_k(0,t)$ and $t\mapsto \psi_k(t-k,t)$ are examples of the function
$
	F(t) = \vert At + B \vert/\sqrt{t}
$
where $A,B \in \R$. For $t \in[k,k+1]$, the function $F_1(t) = \vert At + B \vert$ is a non negative convex function for all $A,B\in \R$ and the function $F_2(t) = \sqrt{t}$ is a positive concave function. Therefore, $F$ is a quasi-convex function on $[k,k+1]$ so, for any $t_0, t_1 \in [k,k+1]$ and any $\lambda \in [0,1]$, $F(\lambda t_0 + (1-\lambda) t_1) \leq \max\{F(t_0), F(t_1)\}$. In particular, $t_0 = k$ and $t_1 = k+1$, hence $F(k + 1 - \lambda ) \leq \max\{F(k), F(k+1)\}$ for all $0 \leq \lambda \leq 1$ so
\[
	\max\{\psi_k(0,t), \psi_k(t-k,t)\} \leq \max\{\psi_k(0,k), \psi_k(0,k+1), \psi_k(1,k +1)\}.
\]

Then, this upper bound on (\ref{e-defpsi1}) is acheived on the vertices of $T_k$.

\end{proof}

}

%%%%%%%%%%%%%%%%%%%%%%%%%%%%%%%%%%%
%%%%%%%%%%%%%%%%%%%%%%%%%%%%%%%%%%%
%%%%%%%%%%%%%%%%%%%%%%%%%%%%%%%%%%%

\section{Proofs of Main Results}

\subsection{Proof of Theorem \ref{T-positive}}
Let $f$ be defined as in (\ref{def-f}), {\color{black} but with scalars that are either non positive or non negative} and $\mathcal{G}$ as in (\ref{eq:wavdict}). Recall the first step is to select $\gamma_0=(t_0, \xi_0, u_0)\in(0,\infty)\times\R^2$ such that $|\langle f,G_{\gamma_0}\rangle|$ is maximized and as a result the norm of the residual $R f = f-\langle f, G_{\gamma_0}\rangle G_{\gamma_0}$ is minimized. We have
\begin{align}\label{eq:pv}
 \langle f,\, G_{\gamma}\rangle &= \int_\R \bigg[\sum_{j=1}^N a_j\rect(x-j)\bigg]\bigg[\frac{1}{\sqrt{t}}e^{2\pi i\xi x}\rect\bigg(\frac{x-u}{t}\bigg)\bigg]  \, dx \notag \\
&= \frac{1}{\sqrt{t}}\sum_{j=1}^N a_j\int_\R \rect(x-j)\rect\bigg(\frac{x-u}{t}\bigg)e^{2\pi i\xi x}  \, dx .
\end{align}

Let  $k  <t\leq k+1$, with $k\ge 1$ (we've seen the case of $k=0$ in Lemma \ref{L-model-case}). Fix $u\in\R$ and choose  $n\in\Z$ for which   
	\begin{equation}\label{deff1}
		\bigg[u-\frac t2, u+\frac t2\bigg]=\bigg[n-\frac 12-s,\ n -\frac 12 + t-s\bigg]
	\end{equation}
with  $0\leq s\leq t-k $; see Fig. \ref{fig2}.
  \begin{figure}[hbt!]
  \centering
     \begin{tikzpicture}

 \fill[gray!20] ( .5,0)--  ( .5, 1)-- (2.5,1)--(2.5, 0);
\draw[black] ( .5,0)--  ( .5, 1)-- (2.5,1)--(2.5, 0);
 \draw[black] (-1.5,0)--(3.55,0);
  \draw [black, dashed] (-1,0)--(-1, 1)--(0,1)--(0, 0)--(0, 2)--(1, 2)--(1,0)--(1, 1.5)--(2,1.5)--(2,0)--(2, 3)--(3,3)--(3,0) ;
   \draw[black] ( .5, 1)-- (2.5,1);
   \draw[black] (-1, 1)-- (0,1);
   \draw[black](0, 2)--(1, 2);
   \draw[black] (1, 1.5)--(2,1.5);
   \draw[black] (2, 3)--(3,3);
 
 \draw [black,fill]  ( .5,0) circle [radius=0.05];
 \draw ( .12,0) node [black,    below]   {{\tiny $n-\frac 12\!-s $}};
 \draw ( 1.2,0) node [black,    below]   {{\tiny $n-\frac 12 $}};
 \draw [black,fill]  (2.5,0) circle [radius=0.05];
 \draw (3, 0) node [black,    below] {{\tiny $ n-\frac 12+t-s $}};
 \draw [black,fill]  (1,0) circle [radius=0.05];
 \draw [black,fill]  (2,0) circle [radius=0.05];
 \draw ( 1.7,-.3) node [black,    below] {{\tiny $ n+k-\frac 12$}};
  \end{tikzpicture}
    \caption{Window function with $1  <t \leq 2$ }
    \label{fig2}
 \end{figure}
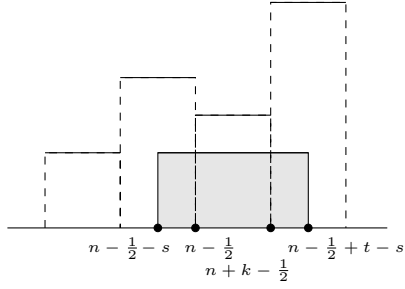
In view of (\ref{eq:pv}), we can write:
 $$
 |\langle f,G_{\gamma }\rangle|=\frac{1}{\sqrt{t}}\bigg\vert a_{n-1}\!\int_{n-\frac 12-s}^{n - \frac{1}{2}}e^{2\pi i \xi x}\, dx + a_{n }\!\int_{n-\frac 12 }^{n + \frac{1}{2}}e^{2\pi i \xi x}\, dx + \dots
 + a_{n+k  } \int_{n + k-\frac 12}^{n -\frac 12 +t-s}e^{2\pi i \xi x} \, dx\bigg\vert.
 $$
 After applying the triangle inequality, which is exact when  $\xi=0$ because the {\color{black} $a_j$'s are of the same sign or zero}, we obtain
 	\begin{equation}\label{psi}
 		|\langle f,G_{\gamma }\rangle| = \frac{1}{\sqrt{t}}\bigg( a_{n-1}s + (a_{n }+... + a_{n+k-1}) + a_{n+k }  (t-s-k )		\bigg)=\psi_k(s,t).
	\end{equation}
 where {\color{black} we retain the notation of $a_j$ despite a potential change of sign}.
 
We now show the maximum of the function $|\langle f,G_{\gamma }\rangle|   = \psi_k(s,t)$ is in the triangle $T_k$ formed by the lines $s=0$, $s=t-k $ and $t=k+1$, with $k\ge 1$. We have:
\[
\max_{\gamma=(t,\xi,u)\in \R^+\times\R^2 } |\l f, G_\gamma\r|= \max_{0\leq k\leq N}  \max_{\gamma \in (k, k+1]\times\R^2 } |\l f, G_\gamma\r|.
\]

Lemma \ref{Lk} shows that when $u\in\R$ and $n\in\Z$ satisfy \eqref{deff1}, the maximum of $ |\l f, G_\gamma\r|$  with $\gamma \in [k, k+1]\times\R^2 $ is as in \eqref{e-mk}.
   
   Our proof shows that $|\l f, G_\gamma\r|$ takes the values in \eqref{e-mk} when $\xi=0$, $u $ is an integer and either $t=k+1$ or $t=k$. From this observation, (\ref{main-max}) follows. 
  \endproof

{\color{black}
\subsection{Proof of Theorem \ref{T-gen-wavelet}}

Let $\{a_j\}_{j=1}^N\subset \R$ and consider $G_{\widetilde{\gamma}} \in \mathcal{G}_{(t,0,u)}$. Take $f$ as in (\ref{def-f}). The proof is nearly identical to the proof of Theorem \ref{T-positive}.
Let  $k  <t\leq k+1$, with $k\ge 1$ since we have done the case for $k = 0$. Fix $u\in\R$ and choose  $n\in\Z$ for which   
	\begin{equation}\label{deff}
		\bigg[u-\frac t2, u+\frac t2\bigg]=\bigg[n-\frac 12-s,\ n -\frac 12 + t-s\bigg]
	\end{equation}
with  $0\leq s\leq t-k $; see Fig. \ref{fig2}.

Since $\xi = 0$, following the proof of Theorem \ref{T-positive} results in a similar calculation to get
 	\begin{align*}\label{psi}
 		|\langle f,G_{\widetilde{\gamma}}\rangle| &= \frac{1}{\sqrt{t}}\bigg\vert a_{n-1}s + (a_{n }+... + a_{n+k-1}) + a_{n+k }  (t-s-k ) \bigg\vert\\
		&=\psi_k(s,t).
	\end{align*}
where the absolute value remains since the sequence may contain negative values. By Lemma \ref{Lk} and similar reasoning to the proof of Theorem \ref{T-positive}, the form (\ref{main-max-2}) follows. 
}

\begin{remark}
{\color{black} The maximization problem with arbitrary real-valued sequence and general waveform dictionary may have optimal modulation parameter not necessarily equal to zero. For example, take a sequence where $a_j=-a_{j-1}$ for each $j=2,\dots,N$. If $a_1=-a\in\mathbb R$ then $a_{2k}=a$, $a_{2k-1}=-a$ and by following (\ref{form}) we get 
\begin{equation}
\label{eq:evaluate}
     \vert\langle f,\, G_{\gamma}\rangle\vert = \frac{|a|}{\sqrt{t}} \left| \sum_{j=1}^N (-1)^j \int_\R \rect(x-j) \rect\bigg(\frac{x-u}{t}\bigg)\, e^{2\pi i\xi x}\, dx\right|.
\end{equation}
Consider the inner integral
\begin{equation}
\label{eq:PHI}
  \Phi(j,\gamma)=  \int_\R \rect(x-j) \rect\bigg(\frac{x-u}{t}\bigg)\, e^{2\pi i\xi x}\, dx.
\end{equation}
Let $\mathcal D_j$ and $\mathcal D_{u,t}$ be the support, respectively, of $ \rect(x-j)$ and $\rect\left(\frac{x-u}{t}\right)$. One of the following four cases applies: (i) $\mathcal D_j \subseteq \mathcal D_{u,t}$;
(ii) $\mathcal D_{u,t} \subseteq \mathcal D_j$;
(iii) $\mathcal D_{u,t} \cap \mathcal D_j\neq \emptyset$ such that $\min \mathcal D_{u,t} \leq  \min \mathcal D_j\leq\max \mathcal D_{u,t} \leq  \max \mathcal D_j$; and (iv) $\mathcal D_j\cap \mathcal D_{u,t} \neq \emptyset$ such that $\min \mathcal D_j \leq  \min \mathcal D_{u,t}\leq\max \mathcal D_j  \leq  \max \mathcal D_{u,t}$.

The function $\Phi$ in these cases can be simplified as follows.}

\textcolor{black}{
\hspace{-15pt}Case (i): $t\geq 1,\ j+\frac12(1-t)\leq u\leq j+\frac12(t-1)$
\begin{equation}
  \Phi(j,\gamma)=  \int_{j-\frac{1}{2}}^{j+\frac{1}{2}} e^{2\pi i\xi x}\, dx 
\end{equation}
Case (ii): $t\leq 1,\ j+\frac12(t-1)\leq u\leq j+\frac12(1-t)$
\begin{equation}
  \Phi(j,\gamma)=  \int_{u-\frac{t}{2}}^{u+\frac{t}{2}} e^{2\pi i\xi x}\, dx
\end{equation}
Case (iii): $t\geq 1,\ j-\frac12(t+1)\leq u\leq j+\frac12(1-t)$
\begin{equation}
  \Phi(j,\gamma)=  \int_{j-\frac{1}{2}}^{u+\frac{t}{2}} e^{2\pi i\xi x}\, dx 
\end{equation}
Case (iv): $t\leq 1,\ j+\frac12(1-t)\leq u\leq j+\frac12(t+1$
\begin{equation}
  \Phi(j,\gamma) =  \int_{u-\frac{t}{2}}^{j+\frac{1}{2}} e^{2\pi i\xi x}\, dx
\end{equation}
We can rearrange them by considering separately $t\leq1$ and $t\geq1$. Then we have, for $t\leq 1$,
\begin{equation}\label{phi1}
\Phi(j,\gamma)=
\begin{cases}
\int_{j-\frac{1}{2}}^{u+\frac{t}{2}} e^{2\pi i\xi x}\, dx,\quad j-\frac12(t+1)\leq u\leq j+\frac12(t-1)\\
\int_{u-\frac{t}{2}}^{u+\frac{t}{2}} e^{2\pi i\xi x}\, dx, \quad j+\frac12(t-1)\leq u\leq j+\frac12(1-t)\\
\int_{u-\frac{t}{2}}^{j+\frac{1}{2}} e^{2\pi i\xi x}\, dx, \quad j+\frac12(1-t)\leq u\leq j+\frac12(t+1)
\end{cases}
\end{equation}
and, for $t\geq 1$,
\begin{equation}\label{phi2}
\Phi(j,\gamma)=
\begin{cases}
\int_{j-\frac{1}{2}}^{u+\frac{t}{2}} e^{2\pi i\xi x}\, dx,\quad j-\frac12(t+1)\leq u\leq j+\frac12(1-t)\\
\int_{j-\frac{1}{2}}^{j+\frac{1}{2}} e^{2\pi i\xi x}\, dx, \quad j+\frac12(1-t)\leq u\leq j+\frac12(t-1)\\
\int_{u-\frac{t}{2}}^{j+\frac{1}{2}} e^{2\pi i\xi x}\, dx, \quad j+\frac12(t-1)\leq u\leq j+\frac12(t+1).
\end{cases}
\end{equation}
}

\textcolor{black}{Let $u=j+\Delta_t$ and, for simplicity, $N=2$.}
\textcolor{black}{
Then, if $t\geq 1$ then we have (\ref{phi1}) so (\ref{eq:evaluate}) becomes:
\begin{equation}
\label{eq:fGgamma1}
\vert\langle f,\, G_{\gamma}\rangle\vert =
\begin{cases}
2\frac{|a|}{\sqrt{t}} \left| \frac{\sin\pi\xi }{\pi \xi} \right|\cdot \left|\sin \left(\frac{\pi}{2} (1 + 2 \Delta_t + t) \xi\right) \right|,-\frac12(t+1)\leq \Delta_t\leq \frac12(1-t)\\
2\frac{|a|}{\sqrt{t}} \frac{\sin^2\pi \xi}{|\pi \xi|}, \quad \frac12(1-t)\leq \Delta_t\leq \frac12(t-1)\\
2\frac{|a|}{\sqrt{t}} \left|\frac{\sin\pi\xi}{\pi \xi}\right|\cdot \left|\sin \left(\frac{\pi}{2} ( 2 \Delta_t -1- t) \xi\right) \right|, \quad \frac12(t-1)\leq \Delta_t\leq \frac12(t+1).
\end{cases}
\end{equation}}

\textcolor{black}{
On the other hand, if $t\leq1$ then we have  (\ref{phi2}) so (\ref{eq:evaluate}) becomes:
\begin{equation}
\label{eq:fGgamma2}
\vert\langle f,\, G_{\gamma}\rangle\vert =
\begin{cases}
2\frac{|a|}{\sqrt{t}} \left| \frac{\sin\pi\xi }{\pi \xi} \right|\cdot \left|\sin \left(\frac{\pi}{2} (1 + 2 \Delta_t + t) \xi\right) \right|,\quad -\frac12(t+1)\leq \Delta_t\leq \frac12(t-1)\\
2\frac{|a|}{\sqrt{t}} \left| \frac{\sin\pi\xi }{\pi \xi} \right|\cdot\left| \sin\pi\xi t \right|, \quad \frac12(t-1)\leq \Delta_t\leq \frac12(1-t)\\
2\frac{|a|}{\sqrt{t}} \left|\frac{\sin\pi\xi}{\pi \xi}\right|\cdot \left|\sin \left(\frac{\pi}{2} ( 2 \Delta_t -1- t) \xi\right) \right|, \quad \frac12(1-t)\leq \Delta_t\leq \frac12(t+1).
\end{cases}
\end{equation}
We observe that: 1. in (\ref{eq:fGgamma2}), $\lim_{t\to0}\vert\langle f,\, G_{\gamma}\rangle\vert=0$; 2. if $\xi=0$, then  $\vert\langle f,\, G_{\gamma}\rangle\vert =0$ in (\ref{eq:fGgamma1}) and (\ref{eq:fGgamma2}). Therefore, $\vert\langle f,\, G_{\gamma}\rangle\vert$ attains its global minimum for $\xi=0$, for all $t>0$. In other words, $\vert\langle f,\, G_{\gamma}\rangle\vert$ is not necessarily maximized when the modulation parameter $\xi = 0$. In fact, if we take $t=2$ and $u$ to be an integer it can be shown that the function $\vert\langle f,\, G_{\gamma}\rangle\vert$ is zero at only countably many values of $\xi$.
}
\end{remark}

%%%%%%%%%%%%%%%%%%%%%%%%%%%%%%%%%%%
%%%%%%%%%%%%%%%%%%%%%%%%%%%%%%%%%%%
%%%%%%%%%%%%%%%%%%%%%%%%%%%%%%%%%%%

 \section{Computational Approach}

{\color{black}

With a full waveform dictionary,  the maximum of $|\langle Rf^n, G_{\gamma} \rangle |$ when $R^nf$ is of mixed sign is not given by Theorem \ref{T-positive}. In fact, after any iteration where we apply Theorem \ref{T-positive} to a non negative or non positive $f$, the residual function $Rf$ (which itself is a step function) used for the next iteration will contain steps of mixed sign. Indeed, the maximized value from the solution will be larger than the average of consecutive sequence elements so the difference will be necessarily of opposite sign on the support of some rectangles and their coefficients. Therefore, Theorem \ref{T-positive} will not apply after the first iteration. We can select a suboptimal $\gamma$ in each iteration that will mimic the form of (\ref{main-max}) and hopefully satisfy (\ref{convmain}) for large enough $0\leq\alpha\leq1$. However, this excludes modulation so we are better off with considering the wavelet dictionary and applying Theorem \ref{T-gen-wavelet}. Nonetheless, the first iteration of the greedy algorithm is simplified provided the sequence of scalars is non positive or non negative.

Now we will utilize Theorem \ref{T-gen-wavelet} to build a simple algorithm that approximates step functions based on real-valued sequences with a wavelet dictionary. Starting with the initial step, consider the subset of the power set of the sequence $\{a_j\}_{j=1}^N$ that contains all subsequences of all lengths such that the terms are consecutive. For example, the sequence $\{2,5,-1\}$ has such subsequences $\{2,5\}$, $\{5,-1\}$, the three singletons, and the entire sequence itself. Then, take the absolute value of the sum of all the elements in each of these subsequences and divide these sums by the square root of the lengths of their corresponding subsequence. Finally, select the maximum of these values and it will be \eqref{main-max-2}.

}

With this procedure, we can recover the $\widetilde{\gamma}_0^*$ for the approximation. The length of the subsequence corresponding to the largest value is $t_0^*$ and, depending on whether $t_0^*$ is even or odd, the $u_0^*$ is the index of the middle term of the sequence or the index of one of the terms of the subsequence tied to be in the middle (ordered according to the index $j\in\N$). This is equal to the maximizer $\widetilde{\gamma}^*$ of (\ref{main-max-2}).

The coefficient $\langle  f , G_{\widetilde{\gamma}_0^*} \rangle$ of the expansion is the sum taken without the absolute value divided by the square root of $t^*$ (could be $k$ or $k+1$ for some integer $k$). Without loss of generality, assume it is the left-most term in the braces of (\ref{main-max-2}). Then, the 1st order greedy expansion of $f$ is 
\begin{align}
f(x) &= \l f(x), G_{\widetilde{\gamma}_0^*} \r G_{\widetilde{\gamma}_0^*} + Rf(x) \\
&=  \bigg(\frac{1}{t_0^*}\sum_{j = n_0^*}^{n_0^* + t_0^*-1} a_j \bigg) \rect\bigg(\frac{x - u_0^*}{t_0^*}\bigg) + Rf(x),
\end{align}
which is a step function over $[n_0^*, n_0^* + t_0^* - 1]$ with one step at the average of the points in the specified interval plus the residuals of taking this approximation.
Then, set $Rf(x) = f(x) - \langle f(x), G_{\gamma_0^*} \rangle G_{\gamma_0^*}$, which is another step function with generic coefficient $R a_j$. Repeat the procedure with $Rf$.

At the $m$th iteration, we consider that same subset of the power set, but of the transformed sequence $\{R^m a_j\}$, which is the original sequence that has gone through $m$ iterations of transformations. Using the same reasoning, we may recover a $\widetilde{\gamma}_m^*$ that results in (\ref{main-max-2}).

If we stop the algorithm after the $m$th step, we are left with 
\begin{equation}
f(x) = \sum_{i = 0}^{m} \langle R^m f(x) , G_{\widetilde{\gamma}_i^*} \rangle G_{\widetilde{\gamma}_i^*} + R^{m+1} f(x) 
= \sum_{j=1}^N R^m a_j \rect(x - j) + R^{m+1} f(x),
\end{equation}
that is, the algorithm shifts the sequence elements depending on the subsequences they belong in as the differences are taken on the scalars and averages resulting from the procedure outlined previously. Dropping the residual term gives us an approximation of $\widehat{f}$ of $f$ based on these shifted sequence elements
\begin{equation}
\hat{f}(x) =  \sum_{j=1}^N R^m a_j \rect(x - j) 
=  \sum_{i = 0}^{m} \langle R^m f(x) , G_{\gamma_i^*} \rangle G_{\gamma_i^*} ,
\end{equation}
which is what we were after. 

{\color{black}
\begin{remark}

This proposed algorithm converges in polynomial time. Recall that the maximized value of $|\langle f, G_\gamma\rangle|$ is some scaled consecutive sequence average.
For each term $a_j$ there are $N$ possible consecutive sequential averages you can calculate which requires approximately $N + (N-1) + \dots + 2 + 1 = N(N+1)/2 \approx N^2$ steps e.g. the sequence containing all elements is $N$ steps for its average, the sequence just containing $a_j$ requires only to record the value. While constructing the set you can compare each new average calculated which adds $N-1$ steps, which is of lower order than the $N^2$. Then we must repeat this $N$ times for each $a_j$ so that the procedure is $O(N^3)$. Finally, for a full run of matching pursuit, we select the number of iterations $m$ so that the order is $O(m\cdot N^3).$
\end{remark}

}

 %%%%%%%%%%%%%%%%%%%%%%%
 %%%%%%%%%%%%%%%%%%%%%%%
 %%%%%%%%%%%%%%%%%%%%%%%

 \section{Empirical Illustration}\label{app} 
 
 In this section, we demonstrate how well the matching pursuit algorithm with a wavelet dictionary of rectangular functions can fit data through examples from simple simulations and time series data. Throughout in the plots, the approximated step function is easy to see as the horizontal red line juxtaposed on the data points. {\color{black} We aim to provide some intuition with these examples   and the fit we are assessing with our method is on the conditional mean of the process our time series is generated by. In addition, the simulations will highlight the potential of our algorithm to approximate time series with discontinuities (breaks) in a conditional expectation function that is (approximately) a step function. }

\subsection{Simulations}

{\color{black}
For the first, we let random process $Y_t$ conditional on a discrete break variable $S_t$ follow a normal distribution:
\[
	Y_t\vert S_t \sim Normal\left( -0.5 I\{S_t = 1\} + 0.1 I\{S_t = 2\}  + 0.5 I\{S_t = 3\}, 0.01\right)
\]
 where $I\{\cdot\}$ is the indicator function and the transition probability matrix for $S_t$ is
\[
\begin{bmatrix}
0.98 & 0.2 & 0 \\
0.005 & 0.98 & 0.015\\
0.02 & 0.08 & 0.90
\end{bmatrix}.
\]
Here the rows represent the current state and the columns represent the next state so for example the probability that $S_{t + 1} = 3$ while $S_t = 2$ is 0.015. Therefore there is some persistence in states $S_t$ so breaks happen slightly infrequently.

In Figure \ref{fig:sim1}, we have randomly generated 250 observations from this distribution and see that our approximation (red) after 11 iterations is able to come close to the true conditional mean (blue shaded line). Notably, it is able to correctly detect the ``spike'' towards the end of the domain as an instance of a true break in the time series.
\begin{figure}[hbt!]
	\centering
	\includegraphics[scale = 0.25]{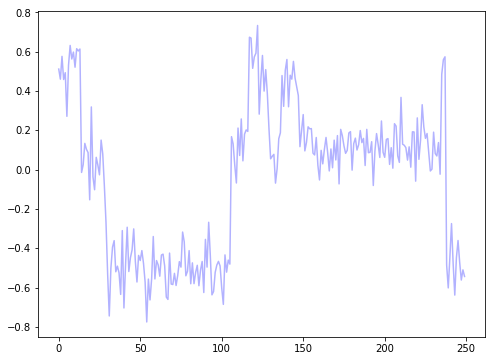}
	\includegraphics[scale = 0.4]{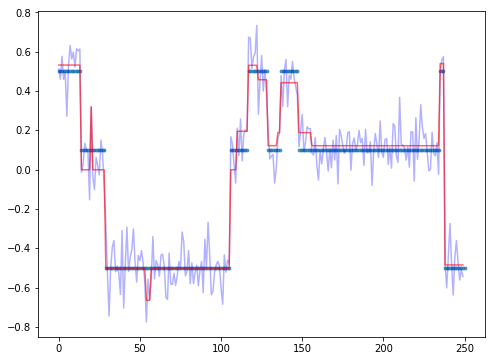}
   \caption{Left: generated time series, Right: approximation on the simulated time series. The red line is the approximation and the horizontal blue lines are the true mean of the process.}
  \label{fig:sim1}
\end{figure}

For our second example we consider
\[
	Y_t\vert S_t \sim Normal\left( -0.4 I\{S_t = 1\} - 0.1 I\{S_t = 2\}  + 0.1I\{S_t = 3\} + 0.4 I\{S_t = 4\}, 0.01\right)
\]
with transition probabilities
\[
\begin{bmatrix}
0.98 & 0.2 & 0 & 0\\
0.02 & 0.95 & 0.03 & 0\\
0 & 0.02 & 0.97 & 0.01\\
0.01 & 0 & 0.02 &0.97
\end{bmatrix}.
\]
In Figure \ref{fig:sim2} we sample 600 observations from this distribution and see the approximation after 21 iterations closely following the conditional mean, although it does not always correctly detect breaks, but it is not too far off, notably a small amount of time after $t=100,200$.
\begin{figure}[hbt!]
	\centering
	\includegraphics[scale = 0.25]{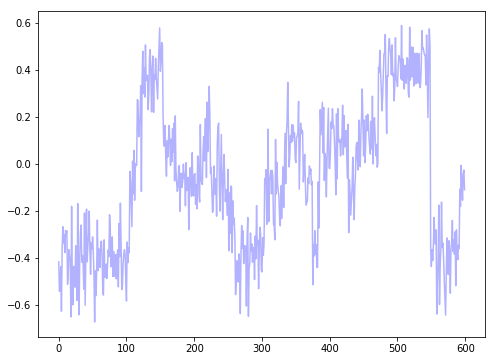}
	\includegraphics[scale = 0.4]{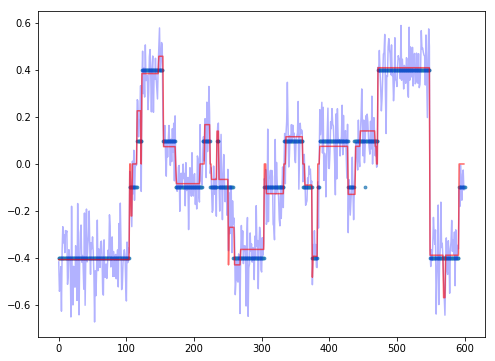}
    \caption{Left: generated time series, Right: approximation on the simulated time series. The red line is the approximation and the horizontal blue lines is the true mean of the process.}
  \label{fig:sim2}
\end{figure}
}

The next simulation is from a set of 500 randomly generated observations from a normal distribution with mean 2 and variance 1 with no breaks. In Figure \ref{fig:normal1},  with one iteration the algorithm fits a step function across the support observations along 2, which is expected since this is the mean.
\begin{figure}[hbt!]
	\centering
	\includegraphics[scale = 0.4]{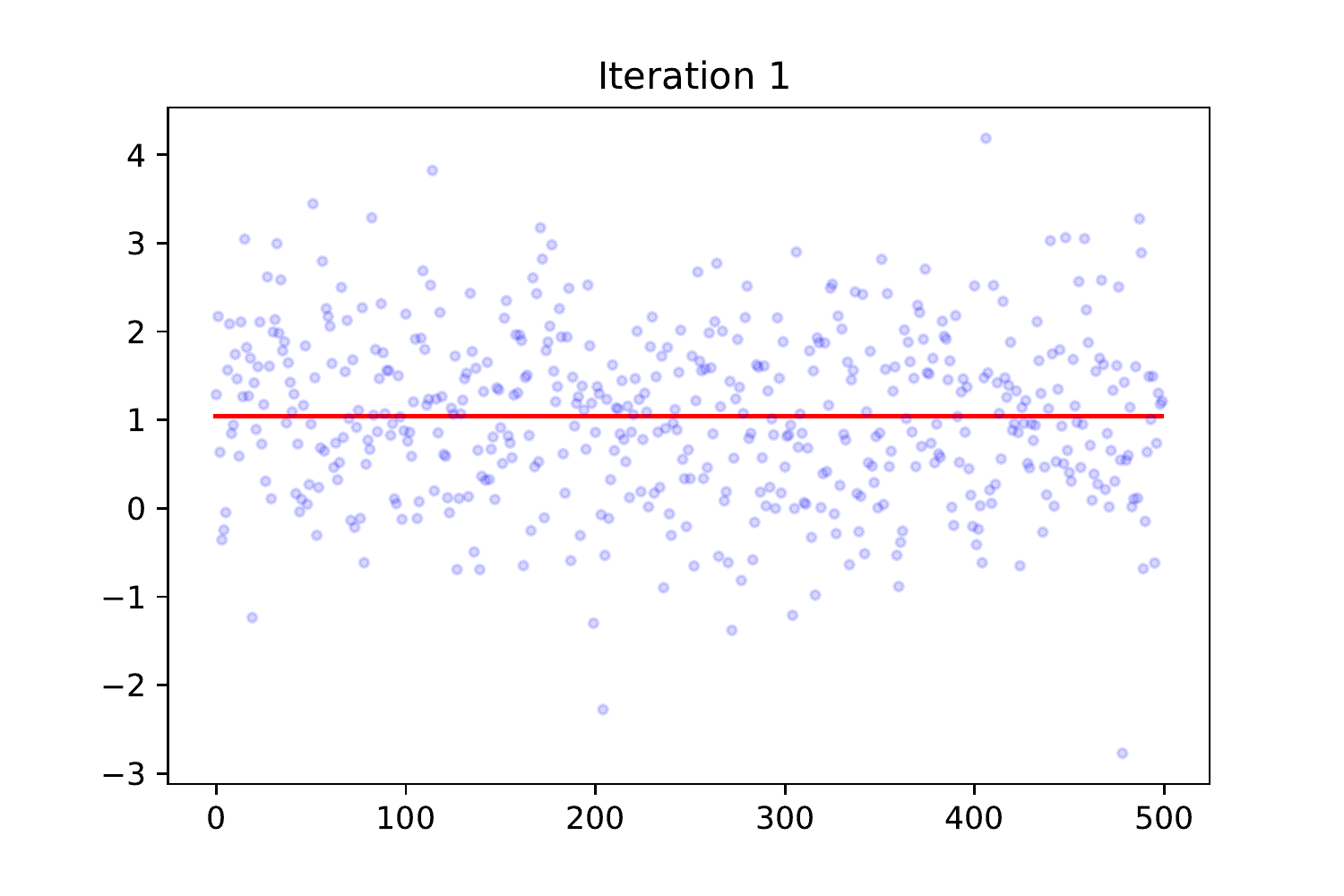}
   \caption{A fit on a sequence of observations from a normal distribution with mean 2 and variance 1.}
  \label{fig:normal1}
\end{figure}
There is an issue that can be seen through \eqref{main-max}. With a standard normal distribution, we should expect a step function across the entire support on zero. However, this always fails since any nonzero element among singleton subsequences will be selected over the ideal step function across the entire support on zero. A simple fix would be to shift the data by constant to avoid this issue. See Figure \ref{fig:normalissue} for an example with 500 randomly generated points.
\begin{figure}[hbt!]
	\centering
	\includegraphics[scale = 0.4]{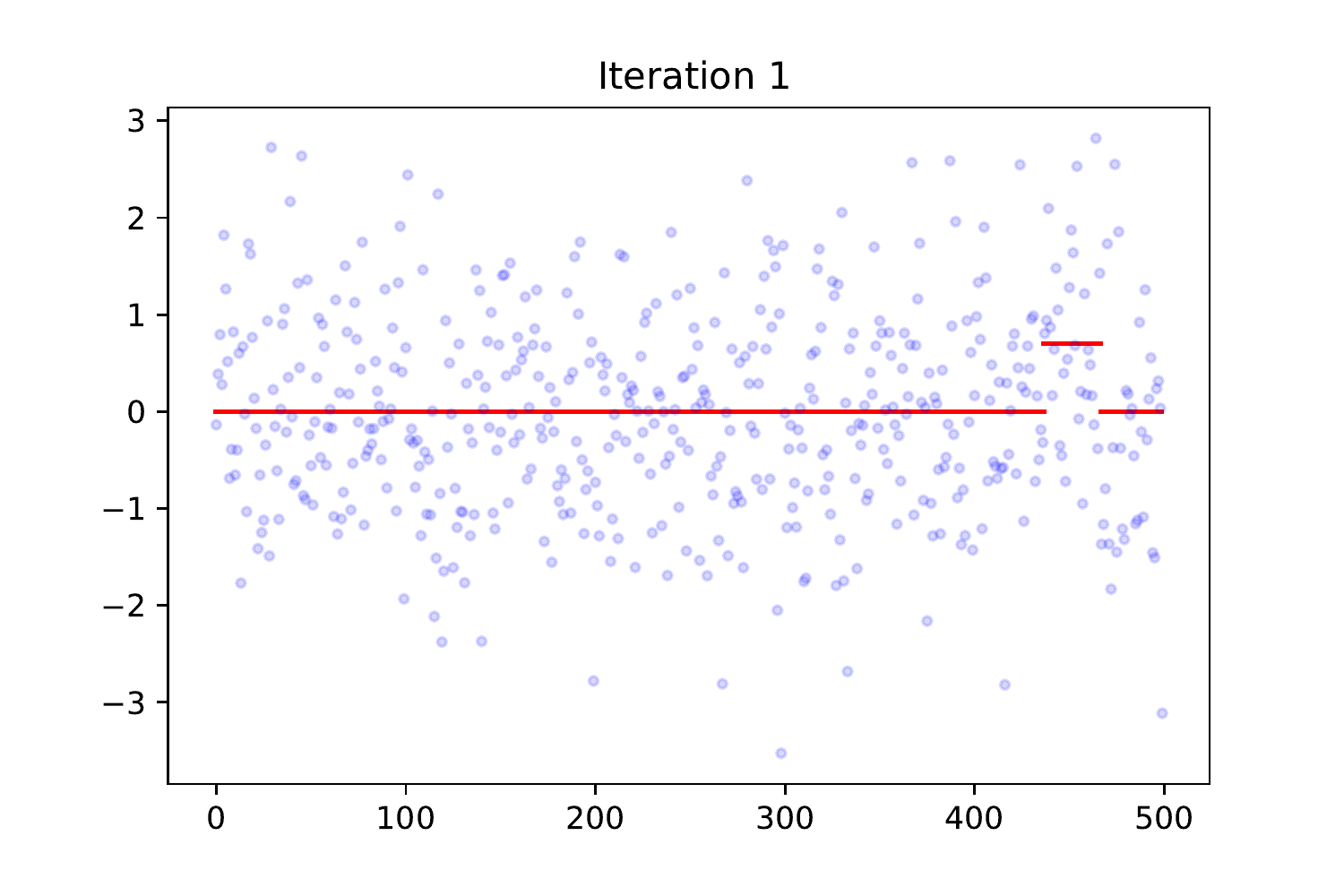}
   \caption{A fit on a sequence of observations from a standard normal distribution.}
  \label{fig:normalissue}
\end{figure}

{\color{black}

In Figure \ref{fig:sim3}, we generate a times series with 100 time periods from a mean-centered AR(2) process with both lag coefficients as 0.3 and standard normal noise term and compare the true conditional mean with the step function approximation after 17 iterations. This is an example of a stationary process and we see that the approximation is susceptible to noisier periods, but otherwise can follow the general tendency of the process albeit missing important persistence information.

\begin{figure}[hbt!]
	\centering
	\includegraphics[scale = 0.5]{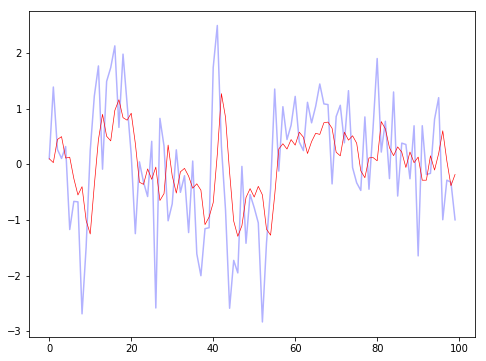}
	\includegraphics[scale = 0.5]{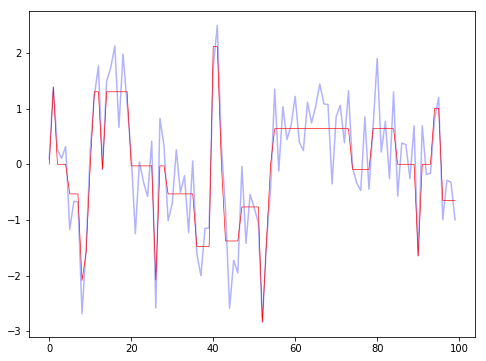}
    \caption{Top: AR(2) TRUE conditional expectation with both lag coefficients as 0.3 (red), Bottom: Step function approximation on the simulated time series. The red line is the approximation and the horizontal blue lines is the true mean of the process.}
  \label{fig:sim3}
\end{figure}

It is worth mentioning that this approximation method is related to simple $k$-means clustering as a way to detect latent structure in the process such as the $S_t$ in the previous examples. In principle, whenever our approximation detects a break it is signaling a change of the time series into a new cluster. This effectively clusters the observations across iterations. However, the main difference is that this approach does not assume the number of clusters and instead searches for consecutive data points that have a significant time mean as a basis to form clusters. This method is highly adapted to this application due to the choice of rectangular window function being able to detect these breaks. 

In Figure \ref{fig:kmeans} we compare our approach as a clustering method to $k$-means assuming the true number of clusters is two where the clusters are normal distributions with mean 0.2 and -0.2. The transition probabilities are symmetric and the probability of remaining in any cluster is 0.97 so the probability of exiting is 0.03. We see that our method is detecting breaks, but it may at times be incorrect, while $k$-means is missclassifying some points, which is expected. A potential advantage of matching pursuit if you are interested in the mean is its ability to mitigate some approximation error since it is directly approximating the mean without knowledge of the number of groups while simple $k$-means is constrained to make an assignment into one of two clusters, potentially including outliers into the mean calculation within clusters. 
\begin{figure}[hbt!]
	\centering
	\includegraphics[scale = 0.5]{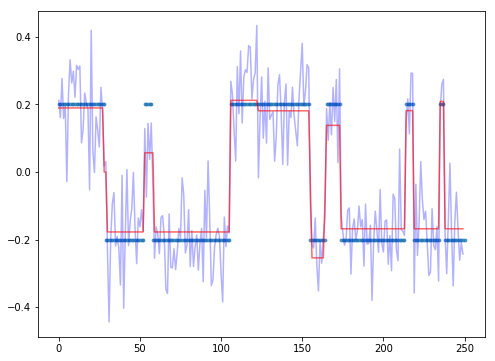}
	\includegraphics[scale = 0.5]{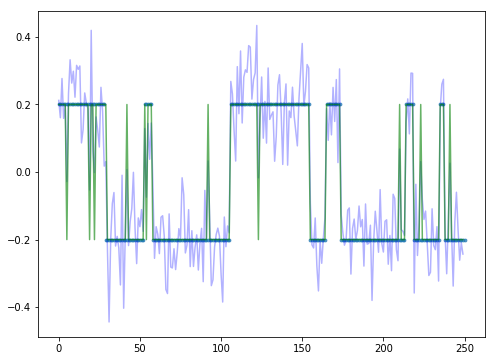}
   \caption{Top: Matching pursuit approximation of the mean. Bottom: $k$means clustering assignments (green), where spikes indicate an incorrect assignment not to be confused with the mean approximation}
  \label{fig:kmeans}
\end{figure}
}

 \begin{figure}[hbt!]
\centering
  \includegraphics[scale=.33]{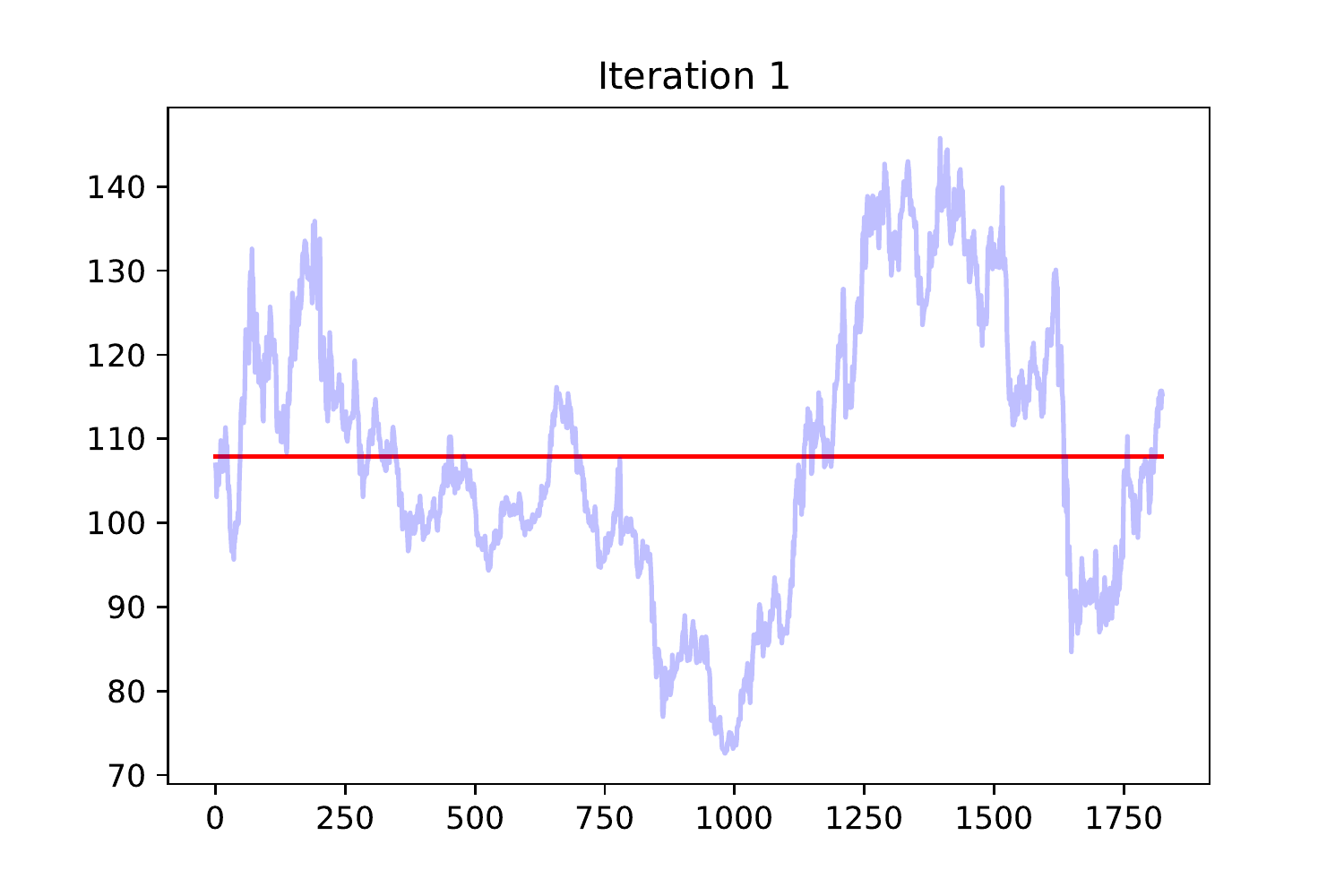}
  \includegraphics[scale=.33]{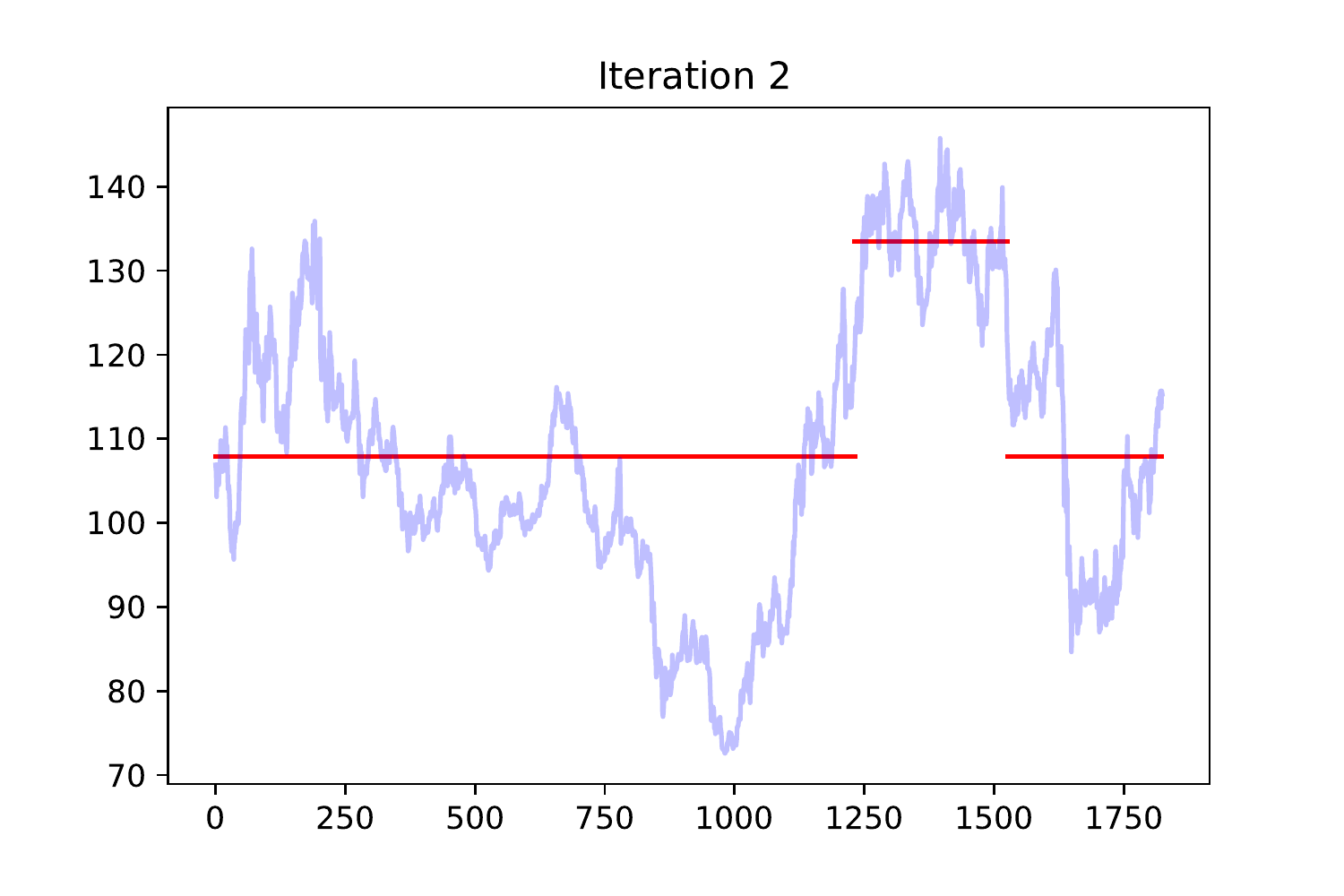}\\
    \includegraphics[scale=.33]{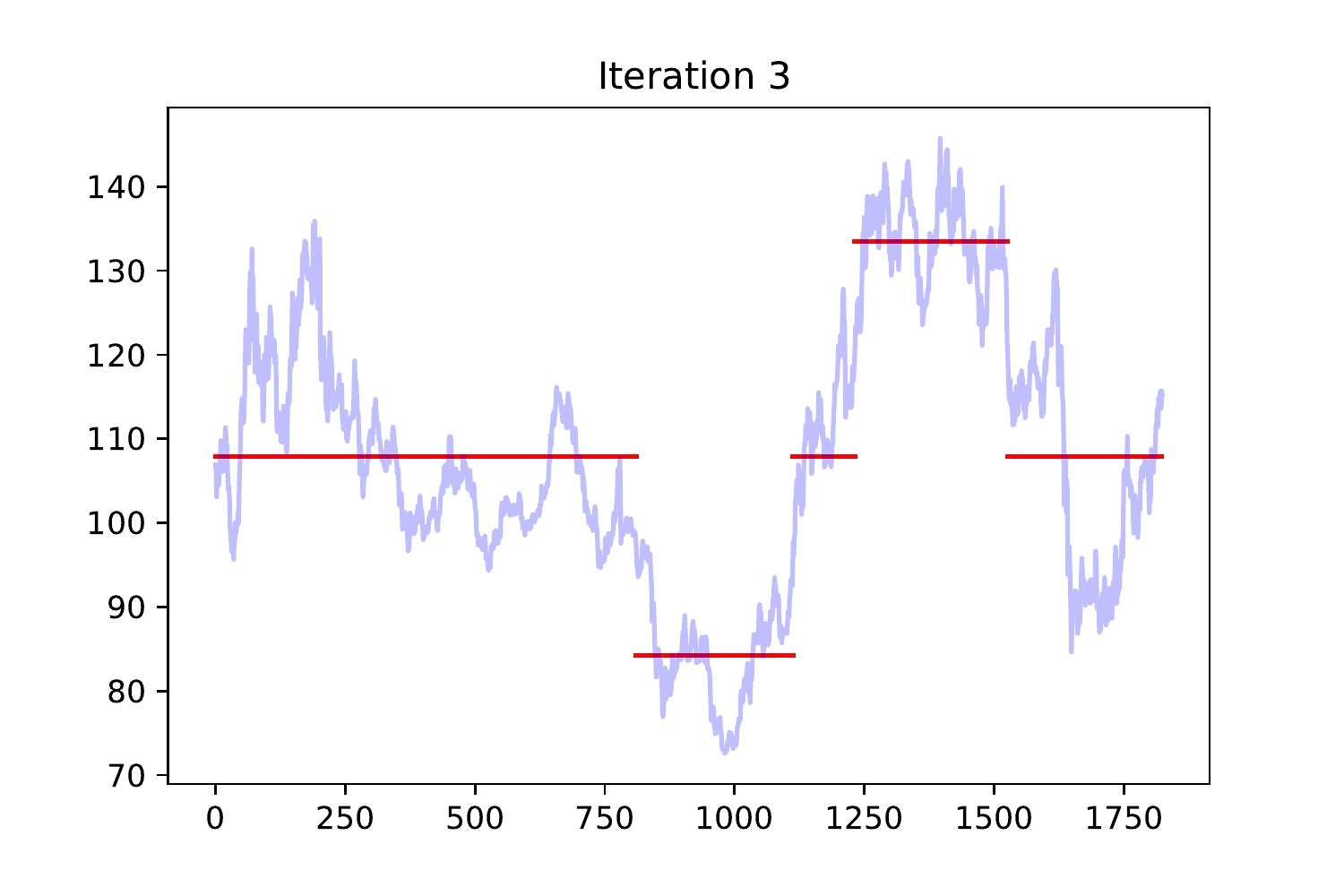}
     \includegraphics[scale=.33]{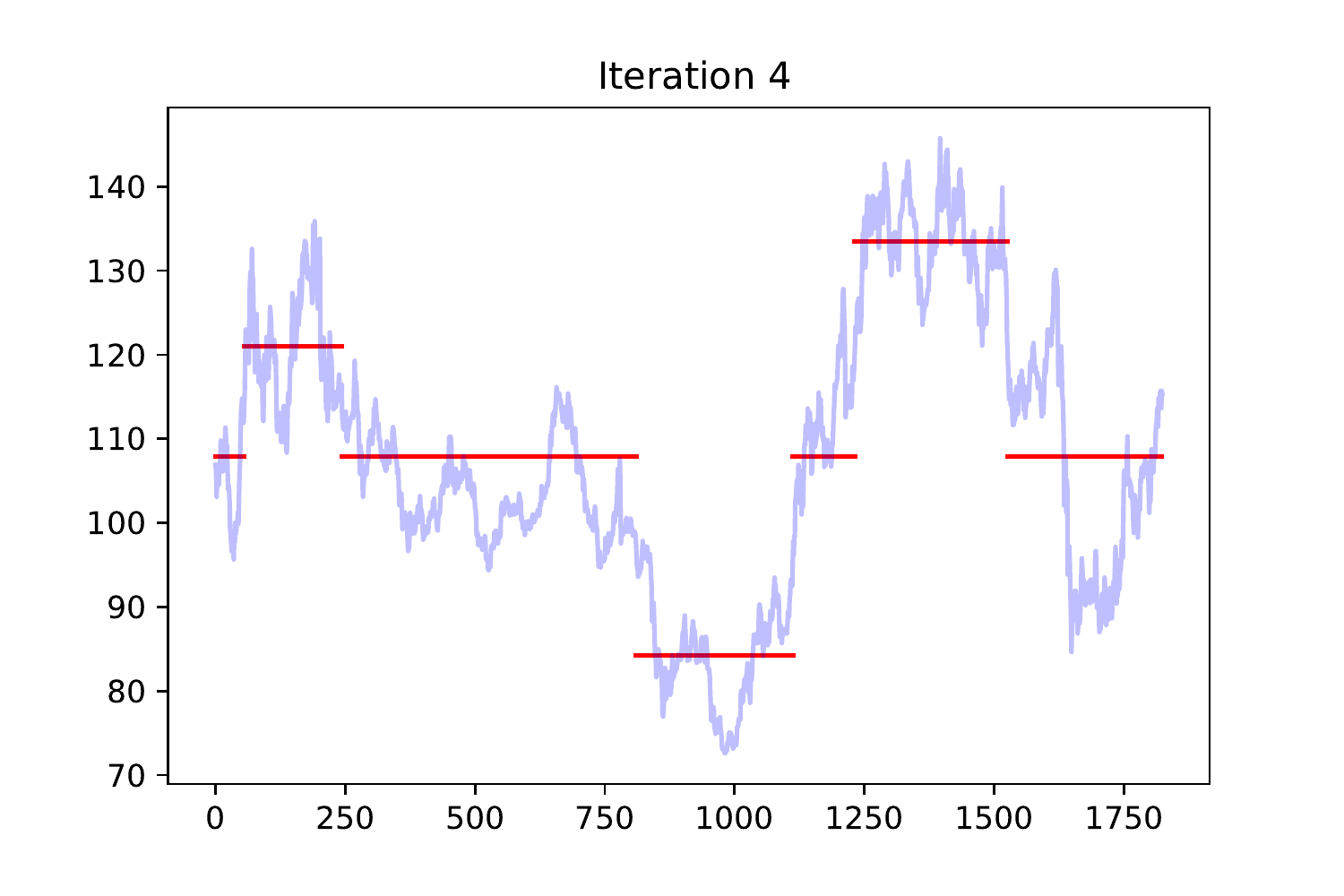}\\
  \includegraphics[scale=.33]{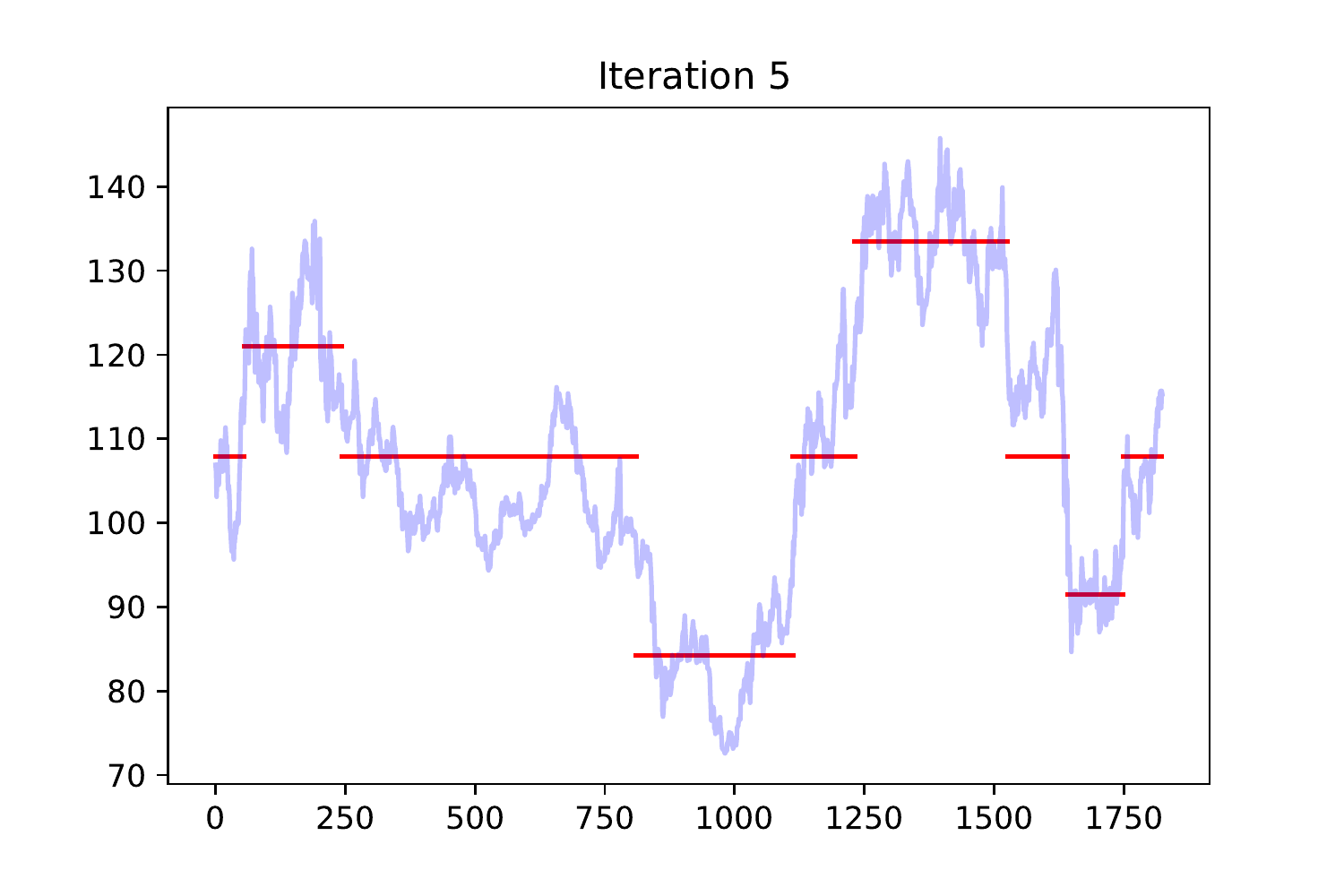}
  \includegraphics[scale=.33]{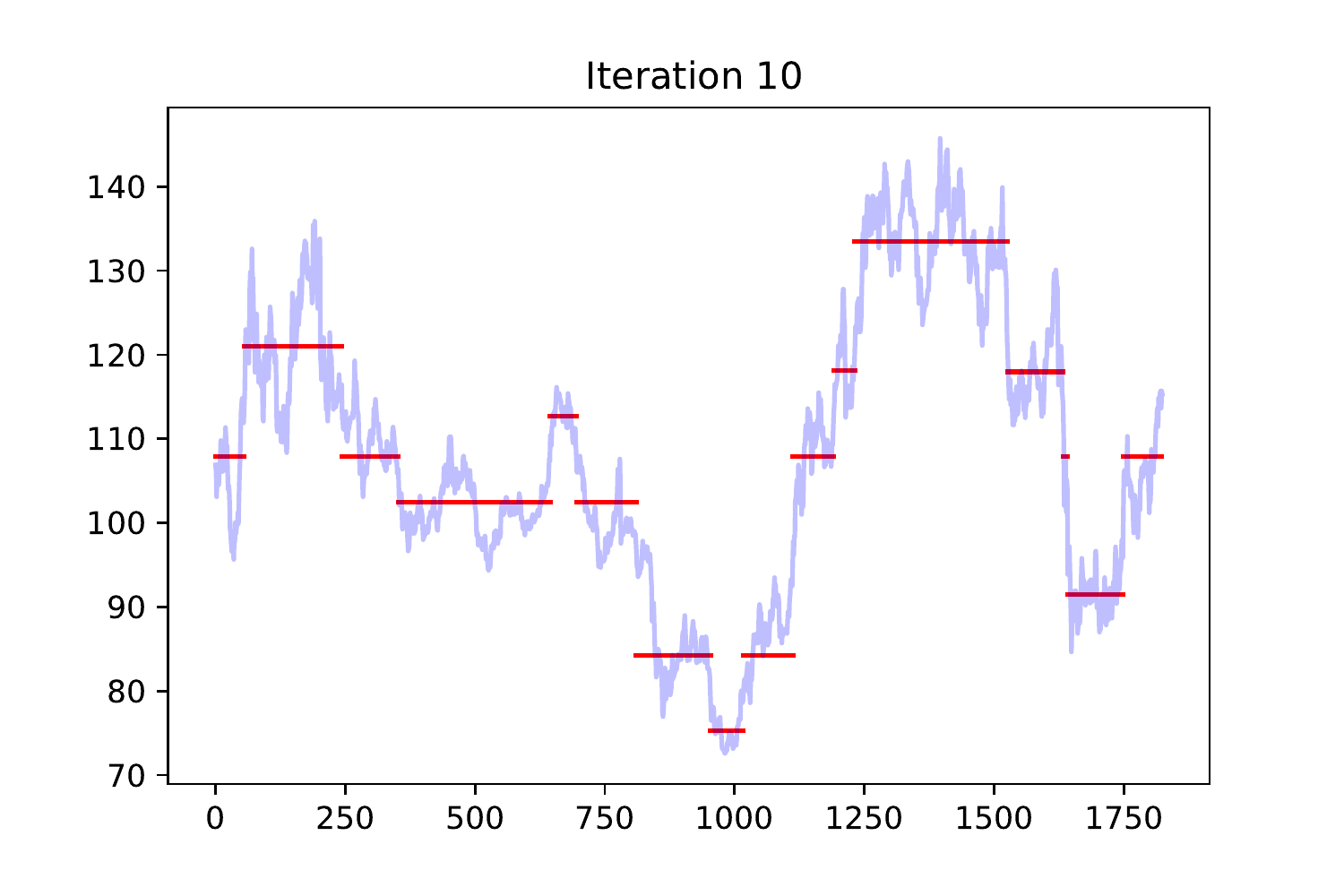}\\
   \includegraphics[scale=.33]{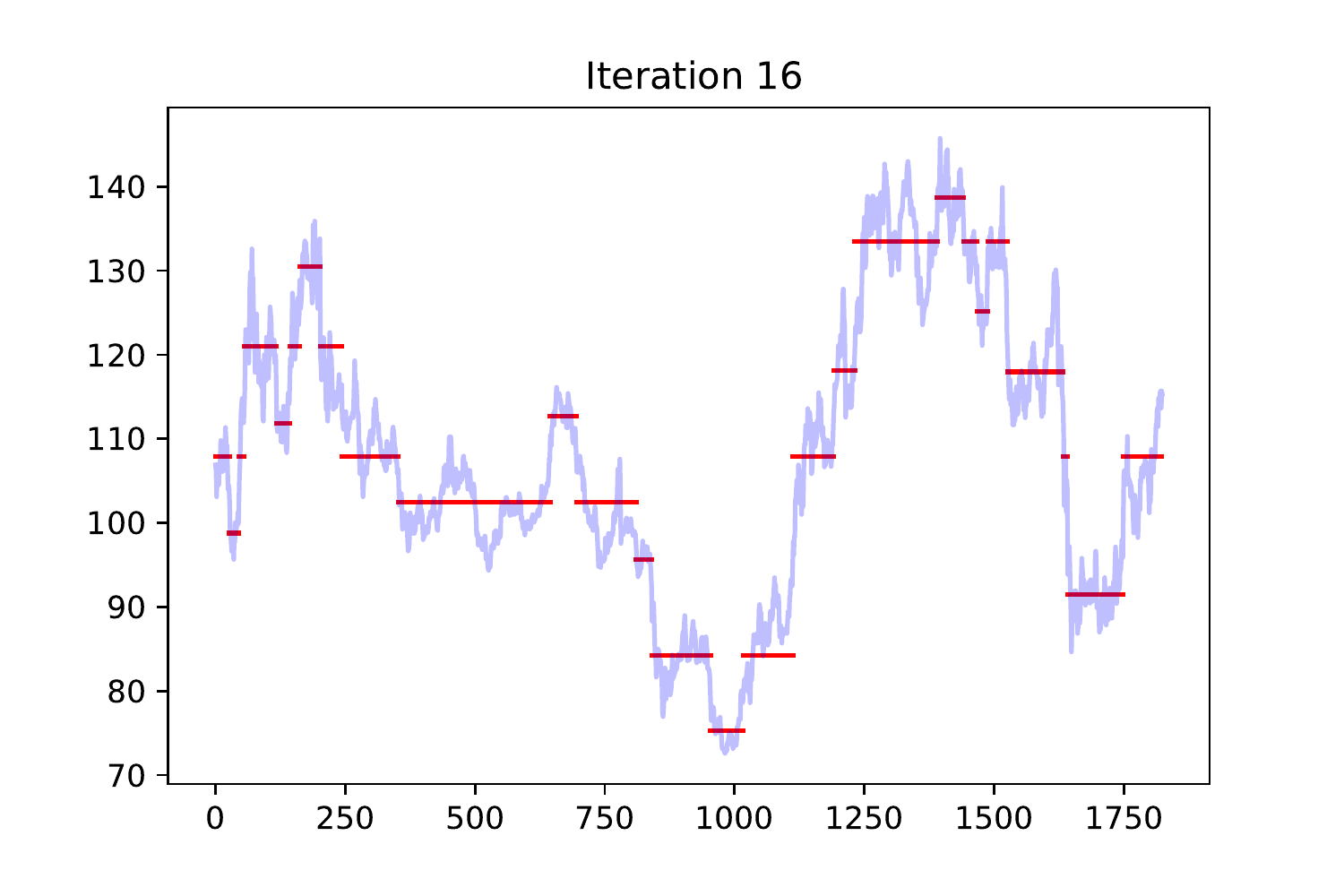}
  \includegraphics[scale=.33]{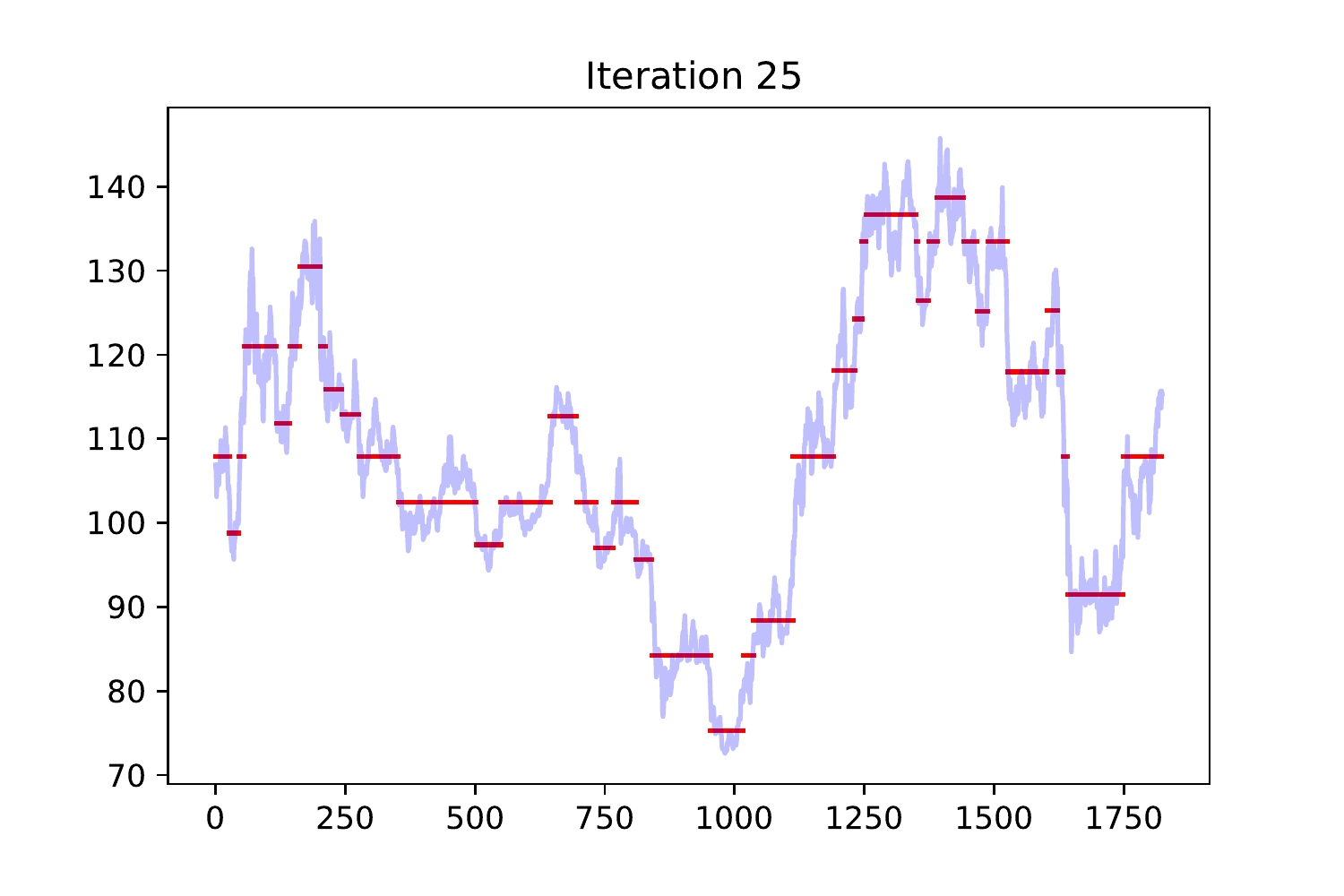}
  \caption{Step function approximation of daily copper prices across 5 years}
  \label{fig:copper}
\end{figure}

\subsection{Real Data Examples}

The rest of our examples are from real time series collected from various sources. We apply the algorithm on time series of daily copper prices, daily number of sales, and daily crude oil prices. The crude oil time series was collected from the website of the U.S. Energy Information Administration (EIA). The copper prices time series was retrieved from the Bloomberg website. The oil prices concern futures contracts and have been observed on the NYMEX market while the time series of copper concerns daily Generic 1st Futures Copper closing prices (HG1 ticker) as exchanged on the COMEX market. Price fluctuations of these two commodities has several important consequences, see e.g. \cite{mastroeni2018reappraisal,mastroeni2018co}. The data set of daily sales was taken from Kaggle, which originated as a forecasting competition. The data set included multiple products at multiple stores across 5 years from 2013 to 2017. For our demonstration, we used the data from only one store and one product to demonstrate seasonality and trend.

With the daily price of copper time series, we showcase the evolution of the approximated step function over 25 iterations of the procedure. In Figure \ref{fig:copper}, the first iteration creates a step function across the entire support of the times of the observations and sets the time mean of the data as its coefficient. As more iterations occur, the algorithm begins to update the step function to fit the data more closely, concentrating on regions of relative stability. {\color{black} While modeling with a step function has its strengths as we've shown, if the data does not originate from a process with conditional mean function that is not approximately a step function, we may miss some important persistence features that are likely to be in this time series. Between days 1500 to 1700, modeling with a step function could be an issue in approximating a step at the steep drop. It is clear that there is an inflection point, but it is not certain if there is a reinforcing effect from the previous periods during the drop, which approximating by a step function will not capture.} Regardless of the absence of bandwidth parameters of our model (being a global approximation), this is a known problem known as \emph{edge bias}; see \cite{racine} for more on local constant regression and \cite{ullah} for derivation of its edge bias. In essence there is not enough data points to make a good assessment along the drop and the data points before and after it will not be informative. 

\begin{figure}[hbt!]
\centering
  \includegraphics[scale=.5]{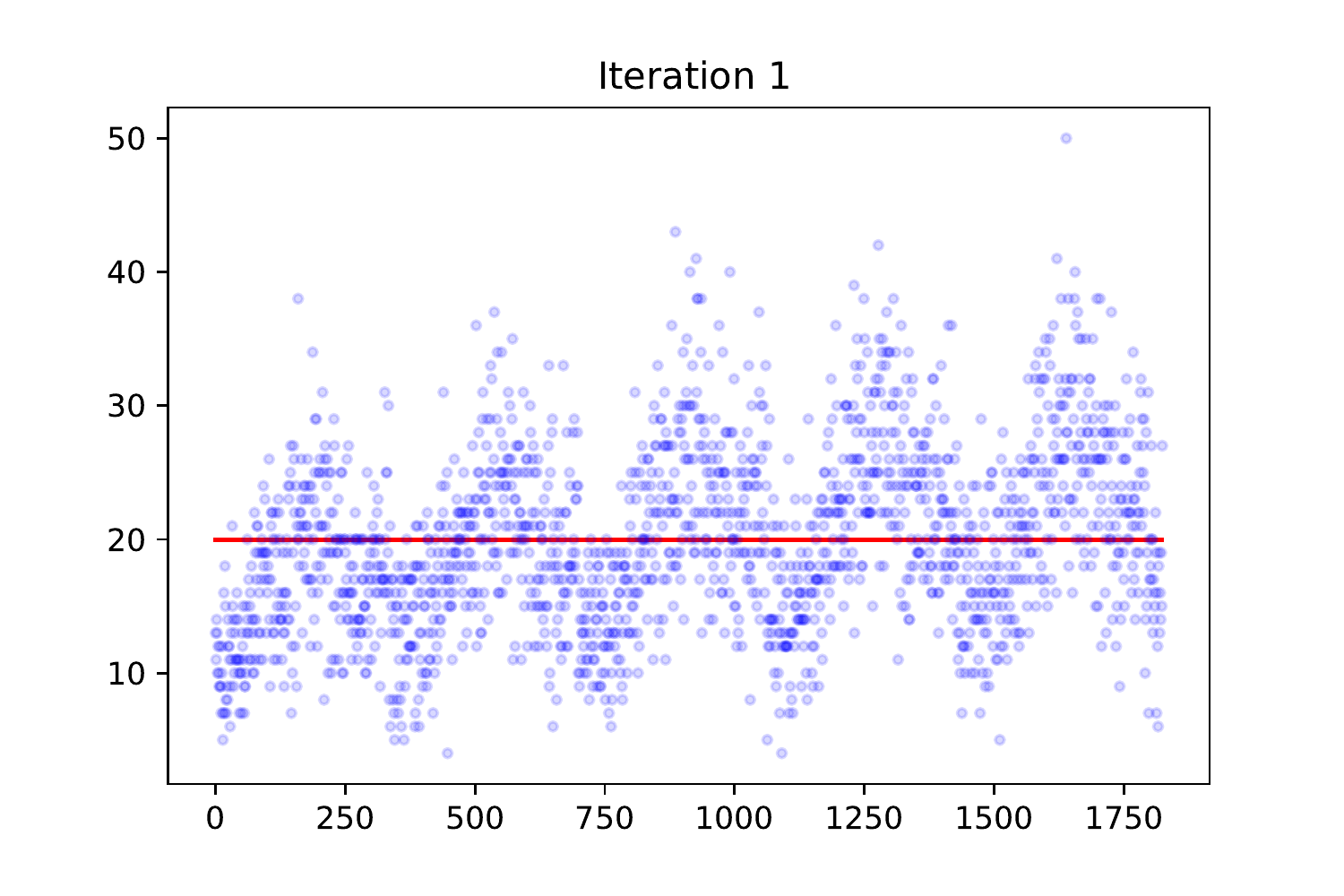}
  \includegraphics[scale=.5]{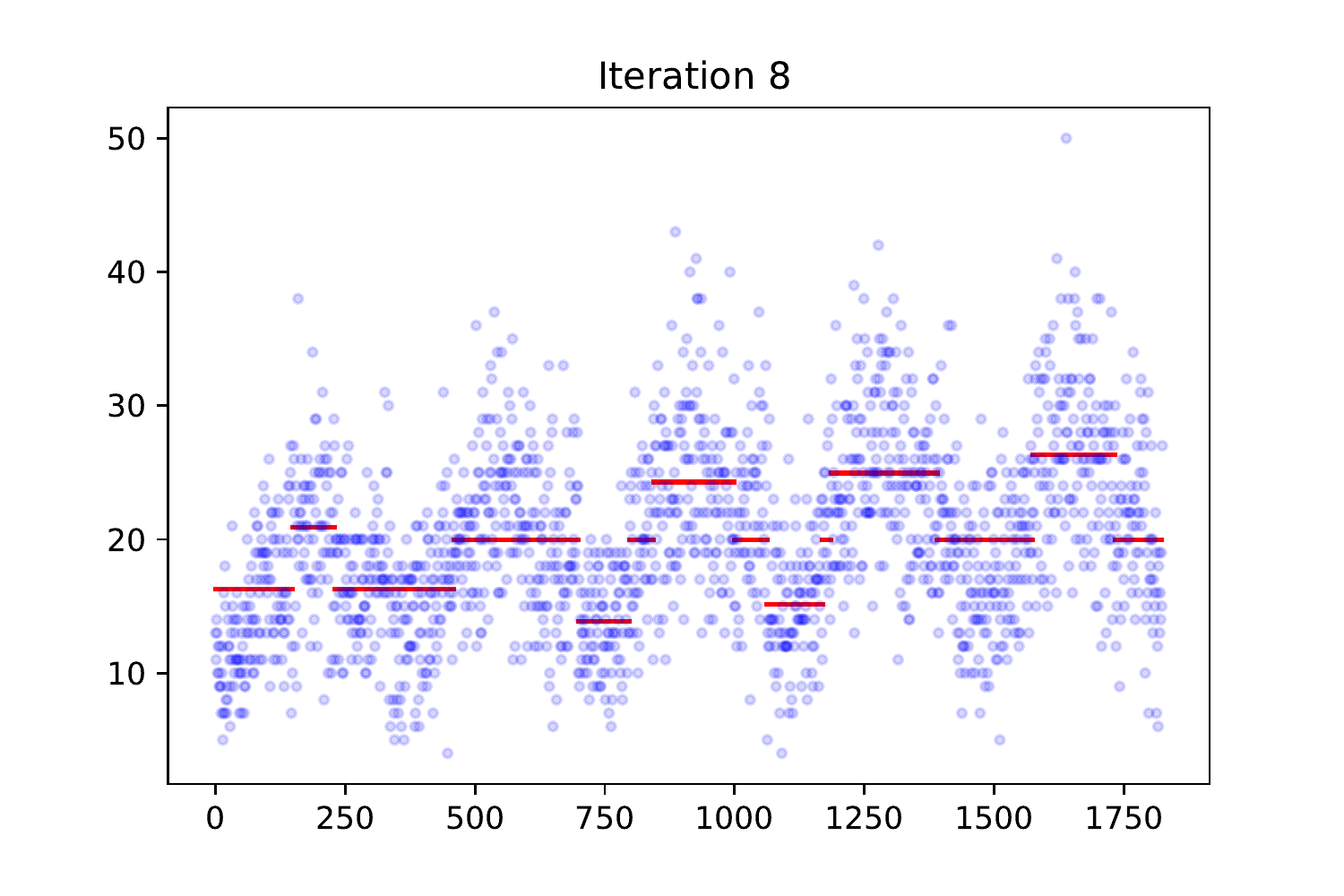}\\
    \includegraphics[scale=.5]{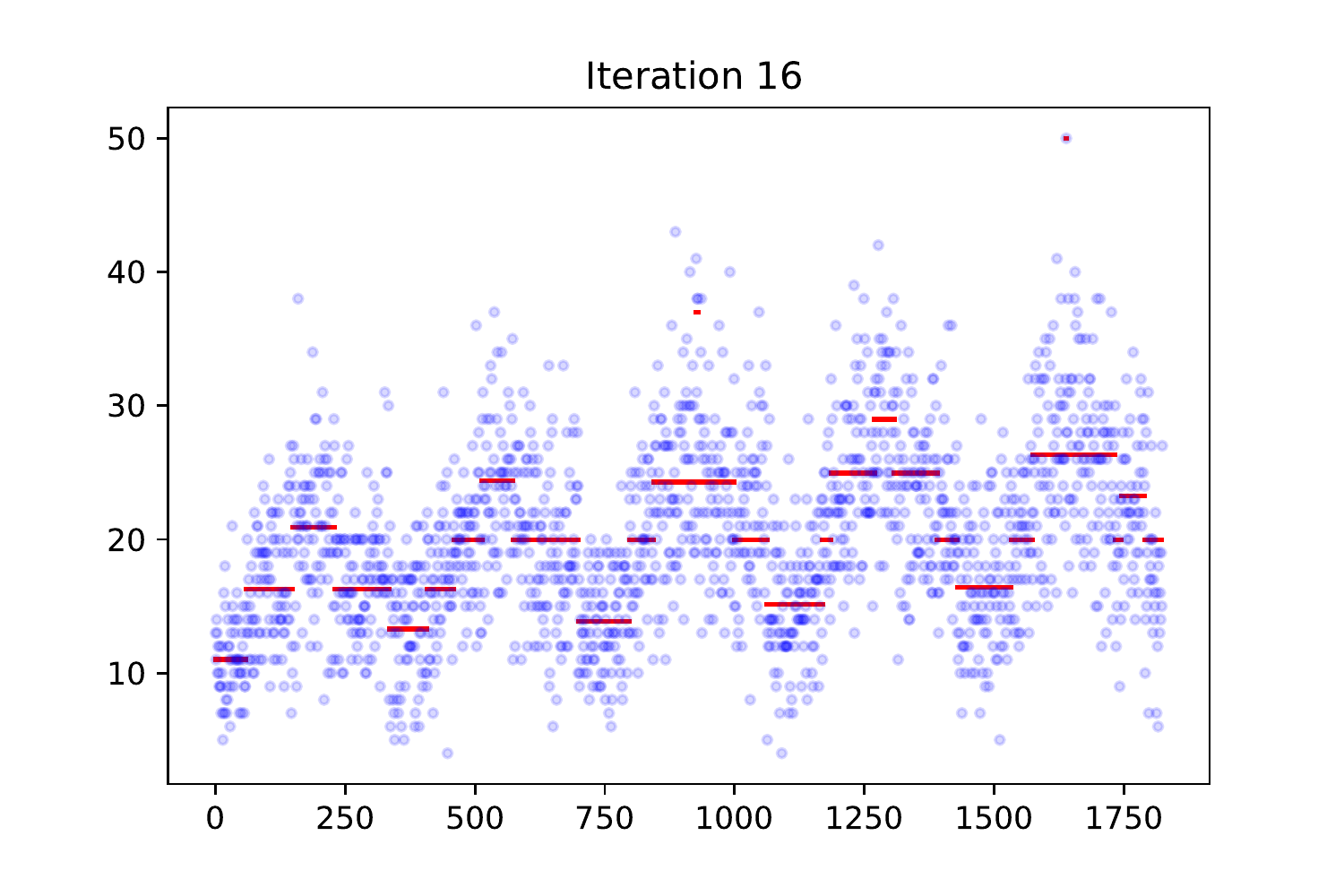}
     \includegraphics[scale=.5]{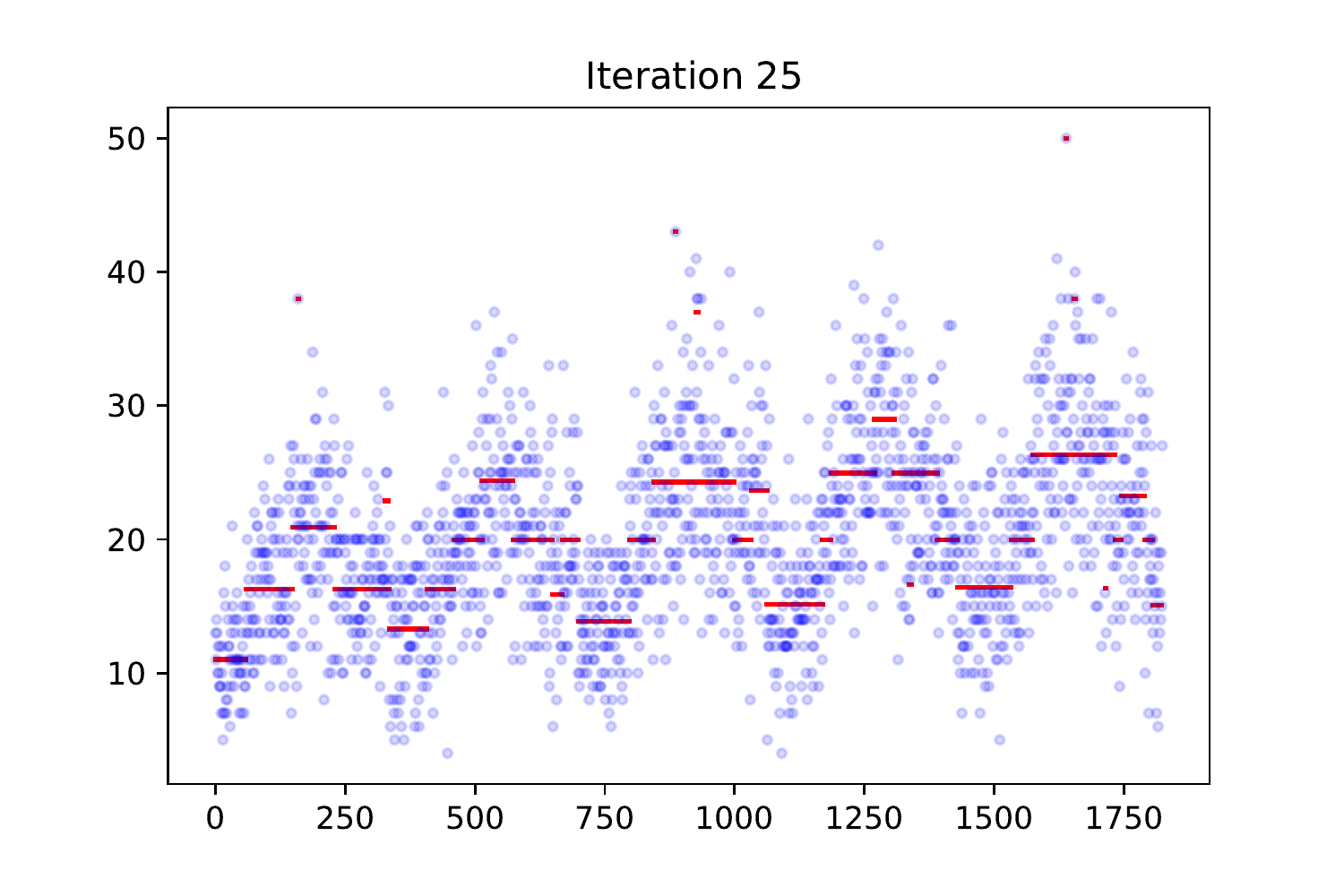}
  \caption{Step function approximation on daily sales across 5 years}
  \label{fig:kaggle}
\end{figure}

In Figure \ref{fig:kaggle}, the daily sales time series displays seasonality and an upward trend, which we do not correct for. The mounds occur between late Spring and early Autumn for reference. As before, the first iteration produces a single large step function across the time mean of the observations. Across future iterations, the steps begin to follow the seasonality and upward trend of the data. Furthermore, the extreme points begin to be fitted with their own step indicating a possible issue with the method. Indeed, at each iteration the singletons are also considered and if an observation is large enough in absolute value, it may dominate the other subsets being considered. {\color{black} However, this pattern occurs yearly and so it might be that this is a true shift and not just a consequence of the variability in the time series. In fact, the time series is quite volatile.} The presence of the volatility renders the naked eye less insightful, but through approximating with a step function we may see regularities in the data and interpret them. Across iterations, the regions become tighter and, with a risk of overfitting, become more precise such as in iteration 25.

\begin{figure}[hbt!]
\centering
     \includegraphics[scale=.5]{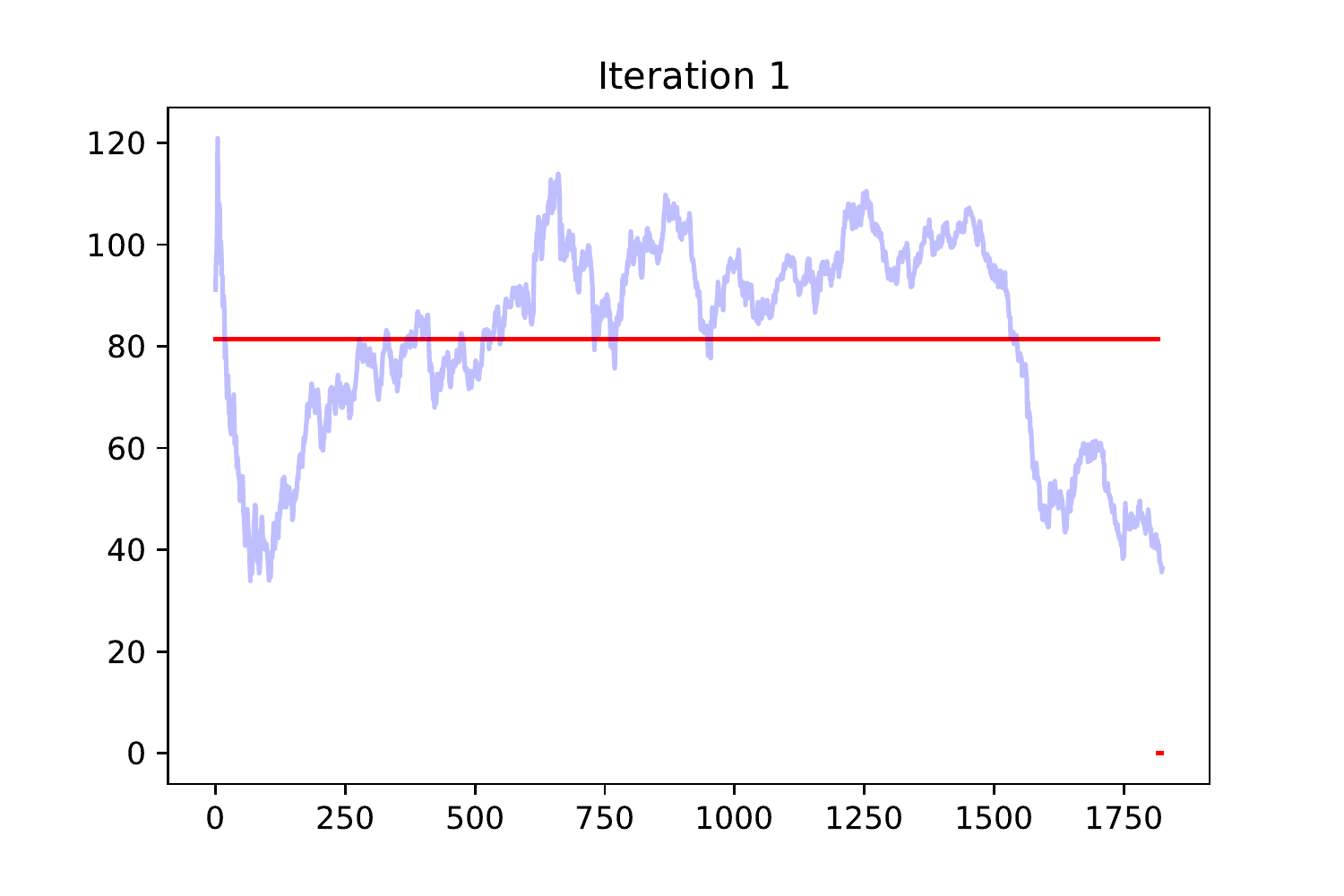}
     \includegraphics[scale=.5]{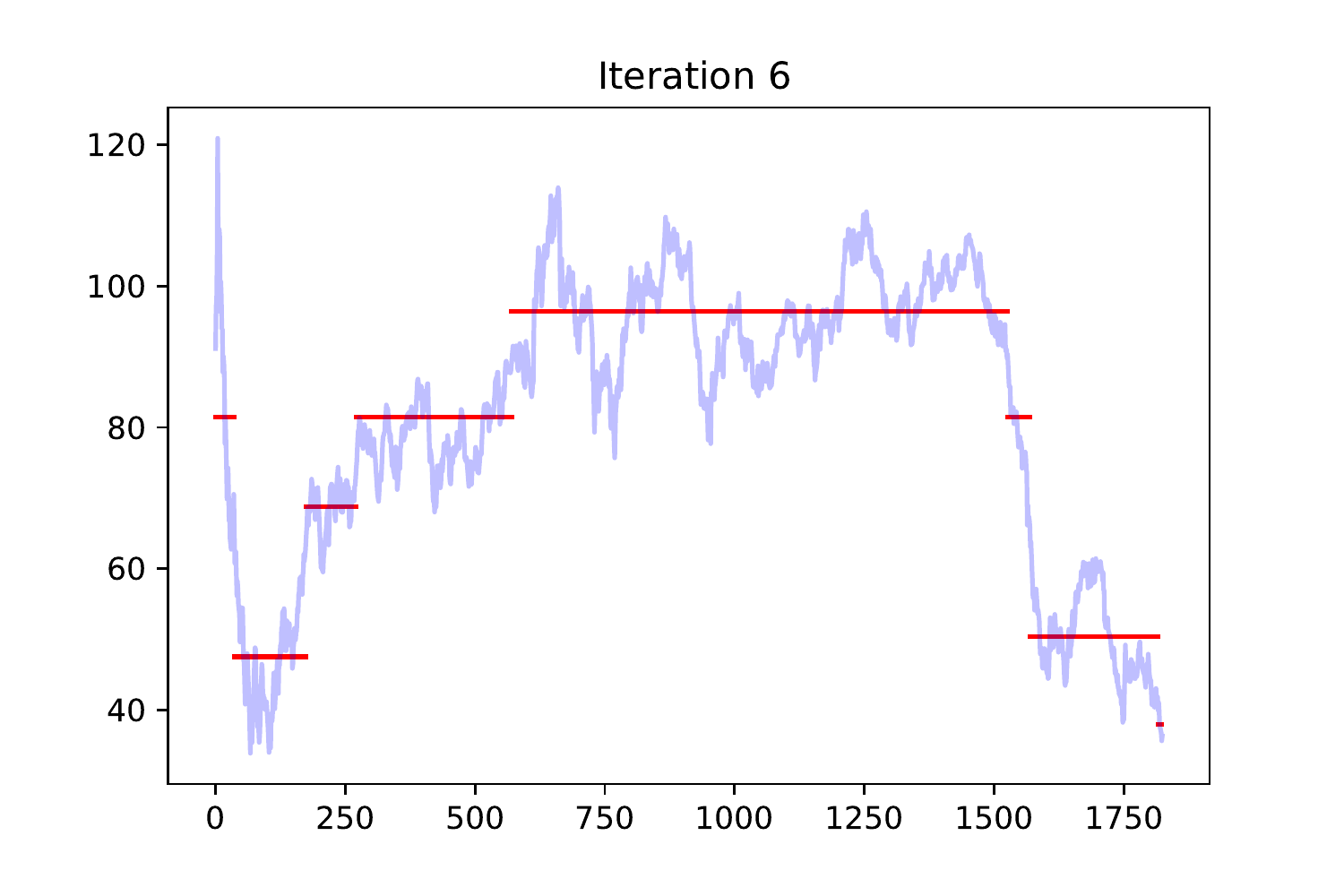}
     \includegraphics[scale=.5]{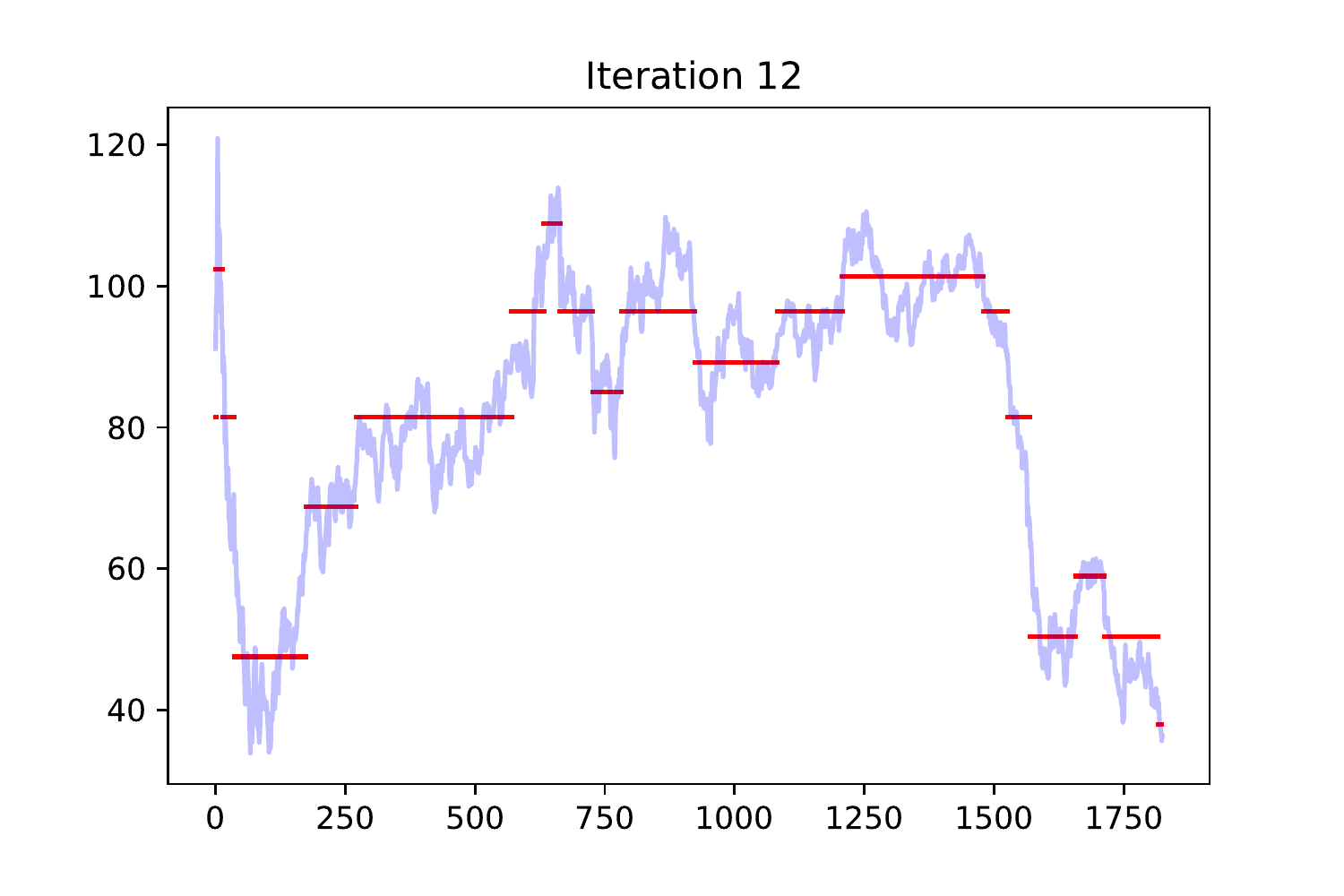}\\
    \includegraphics[scale=.53]{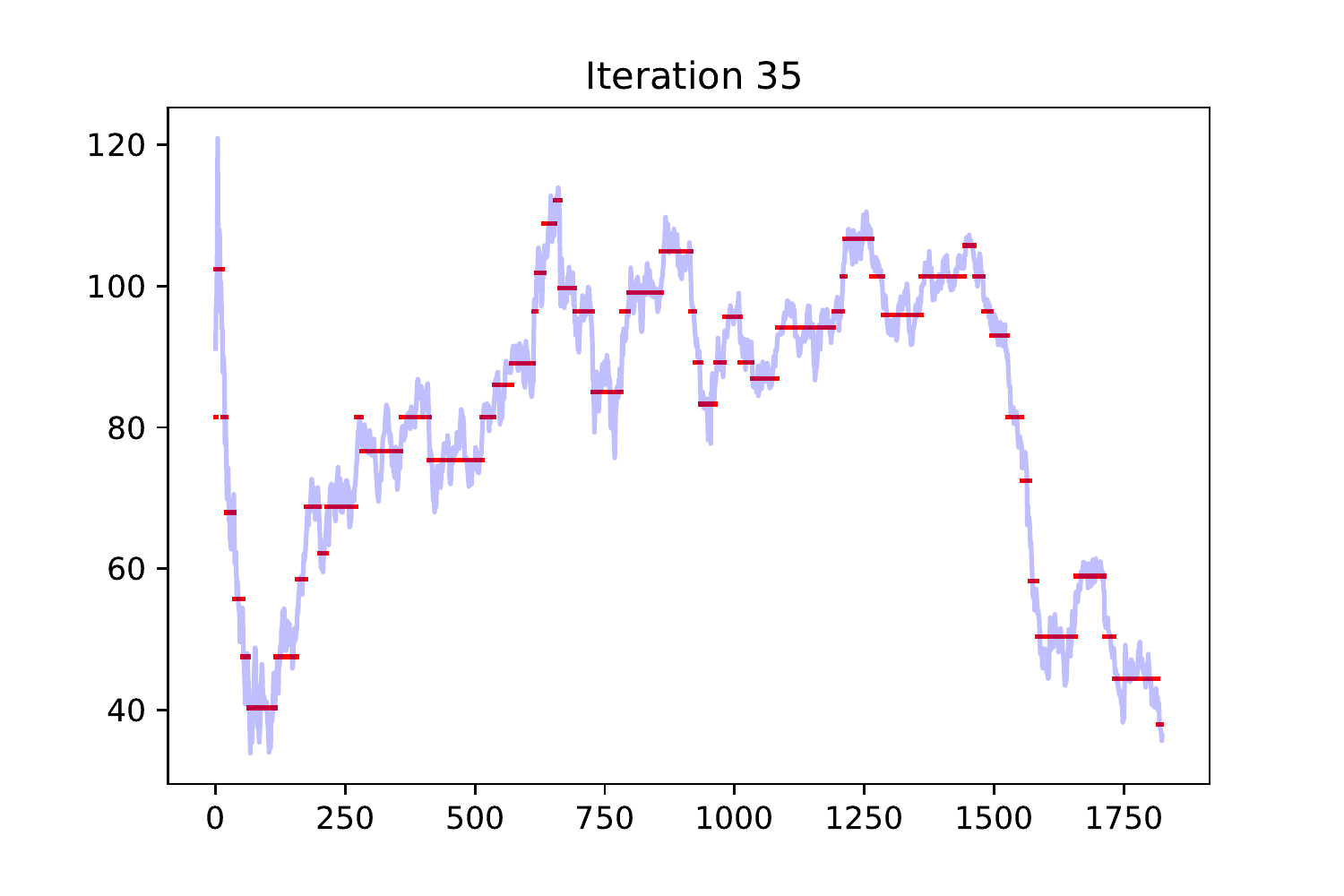}

  \caption{Step function approximation of daily petroleum prices across 5 years}
  \label{fig:petroiter}
\end{figure}

The previous examples initially fit a step function across the entire support of the data, but this is not always the case. In Figure \ref{fig:petroiter}, the daily petroleum prices time series is almost entirely approximated by a positive step in the first iteration except for a single point at the right edge, which is zero. As previously mentioned, this is likely due to the inability of the rectangle function atom to properly explain steep drops. Iteration 6 corrects this, but the discontinuity remains an issue across future iterations as this step is not a convincing fit. Moreover, the steep drop problem is clearly seen on the left portion of the time series and seemingly no amount of iterations with the rectangular atom will correct it. {\color{black} On a positive note, at iteration 35 we can start to see that the step function approximation reveals where the breaks in this time series might be. From here we can take steps to infer properties of the break process within the petrol price process.}

% Dates where extreme data points occur in sales time series:
% 		June 6, June 28, July 7 

 %%%%%%%%%%%%%%%%%%%%%%%
 %%%%%%%%%%%%%%%%%%%%%%%
 %%%%%%%%%%%%%%%%%%%%%%%

\section{Conclusions}
The matching pursuit approximation of signals with a waveform dictionary is an intensive procedure, but may be simplified by considering rectangular window functions. With the smaller wavelet dictionary, a set of coefficients that produce an approximate orthogonal expansion with dictionary elements can be iteratively constructed in polynomial time. In addition, the expansion is simple to interpret as it is a step function with the average of terms in the sample over the support of the different steps. For the general waveform dictionary with rectangular window, one is able to simplify the first step if a representation localized in time and frequency is desired, provided the data points are non positive or non negative.

Comparisons between the result of this method as a global procedure and local methods may be drawn. In possible applications, our method allows for dynamic calculations on partitions rather than choosing partitions or bandwidths based on a rule. Examples of rules usually involve sample estimates of the distribution parameters of the data, which depend on the sample size. This method does not rely on rules and instead allows the data to inform the researcher of the underlying structure. This is important in the clustering setting for time series, where this method is highly adapted to clustering observations when clusters are of constant mean. For example, the popular $k$-means algorithm requires knowledge of the true number of clusters while our algorithm does not.

The assumption \eqref{def-f} on the signal's regression function is reasonable for time series that are measured in fixed length intervals from a continuous process. The discontinuities from the wavelet dictionary with rectangular atoms shows promise to detect breaks in the time series so long as the true regression function is approximately a step function due to the property of being localized in time. Related to this it shows the ability to distinguish periods of regularity in the data even if the signal is not piecewise constant. Still, we require a derivation of statistical properties and sample performance of this method before we can make any meaningful comparisons. Nonetheless, we have established the mathematical properties, facilitating further developments with matching pursuit approximations on data.

\bibliographystyle{unsrt}  
\bibliography{references}

\end{document}